\pdfoutput=1
\documentclass[11pt,leqno]{article}
\usepackage[T1]{fontenc}
\usepackage{authblk}
\usepackage[margin=1.25in]{geometry}
\usepackage{amsmath,amssymb,amsthm,mathtools,booktabs,longtable,array}
\usepackage{dsfont,cite,microtype,hyperref}
\hypersetup{colorlinks,breaklinks,linkcolor=blue,citecolor=blue,urlcolor=blue,
 pdftitle={An explicit solution of the five-expert prediction PDE},
 pdfauthor={Jeff Calder and Nadejda Drenska}}

\newcommand{\R}{\mathbb{R}}
\renewcommand{\v}{\mathbf{v}}
\newcommand{\one}{\mathds{1}}
\renewcommand{\O}{{\mathcal O}}
\newcommand{\eu}{\mathrm{e}} %
\renewcommand{\epsilon}{\varepsilon}
\renewcommand{\bar}[1]{{\overline{#1}}}
\renewcommand{\tilde}[1]{\widetilde{#1}}
\renewcommand{\phi}{\varphi}

\DeclareMathOperator{\csch}{csch}
\DeclareMathOperator{\sech}{sech}
\DeclareMathOperator{\arccoth}{arcoth}
\DeclareMathOperator{\atanh}{artanh}
\newcommand{\diff}{\mathrm{d}} %
\newcommand{\dd}{\,\diff t}
\newcommand{\Iop}{\mathcal I}
\newcommand{\Kop}{\mathcal K}
\newcommand{\Lop}{\mathcal L}
\newcommand{\Top}{\mathcal T}
\newcommand{\Sop}{\mathcal S}

\newtheorem{theorem}{Theorem}
\newtheorem{proposition}[theorem]{Proposition}
\newtheorem{lemma}[theorem]{Lemma}

\theoremstyle{definition}
\newtheorem{remark}[theorem]{Remark}
\numberwithin{equation}{section}
\numberwithin{theorem}{section}

\title{An explicit solution of the five-expert prediction PDE and the exact optimality set of COMB\thanks{{\bf Funding:} JC acknowledges funding from an Albert and Dorothy Marden Professorship, a Simons Fellowship, and National Science Foundation Grant DMS:2436333. ND acknowledges funding from National Science Foundation Grant DMS:2407839. {\bf Source Code:} The computational supplement is archived at \url{https://doi.org/10.5281/zenodo.22723924}. {\bf Formalization:} A Lean~4 formalization of the verification, which proves both main theorems apart from the viscosity characterization of Theorem~\ref{thm:main}, is archived at \url{https://doi.org/10.5281/zenodo.23888815}. The proof development and the preparation of this paper were assisted by ChatGPT and Claude.  The authors take full responsibility for the contents of the paper.}}

\author[1]{Jeff Calder}
\affil[1]{School of Mathematics, University of Minnesota}
\author[2]{Nadejda Drenska}
\affil[2]{Department of Mathematics, Louisiana State University}

\begin{document}
\maketitle
\begin{abstract}
In this paper, we derive an explicit solution of the stationary prediction with expert advice PDE for five experts.  The formula is given in three regions. In the first two regions, it is the four-expert solution plus a single integral with an elementary positive density. In the third region, it is a finite sum of hyperbolic products whose coefficients are determined by one scalar quadrature. 
 Our formula establishes that the rank ordered adversary $(1,0,1,0,0)$ achieves the Hamiltonian maximum in the PDE everywhere in the ordered sector, while the COMB strategy $(1,0,1,0,1)$ is optimal only on a lower dimensional subset of the sector (where $x_1=x_2$ and $x_3=x_4$). 
 The verification of the Hamiltonian inequalities is a tedious task, part of which is completed with a computer assisted proof. The verification reduces to 21 scalar inequalities, which we prove using 147 exact rational Bernstein polynomial certificates.  The exact certificates and their independent arithmetic checks are included in a supplement to this paper, and a Lean 4 formalization machine-checks the verification and both main theorems, apart from the viscosity characterization.
\end{abstract}

\section{Introduction}\label{sec:introduction}

Prediction with expert advice is one of the oldest problems in online learning. A player (an ``investor''), observing the past performance of $n$ experts, repeatedly chooses which expert to follow; an adversary (a ``market'' or ``nature'') simultaneously controls and decides which experts are correct. The player's performance is measured by a function of regret, defined as the gap between the best expert's cumulative gain and the player's own gain. The problem originates with Cover~\cite{Cover65} and Hannan~\cite{Hannan57}. Cover solved the two-expert case exactly and established the $\O(\sqrt T)$ regret rate over $T$ rounds. The subsequent development, surveyed in Cesa-Bianchi and Lugosi~\cite{CBL06}, concentrated on algorithms with good worst-case guarantees rather than exact minimax play: the weighted-majority and aggregating strategies of Littlestone--Warmuth~\cite{LW94} and Vovk~\cite{Vovk90}, and more generally the multiplicative-weights (Hedge) family~\cite{FS97,CFHHSW97}, achieve regret $\O(\sqrt{T\log n})$, and Haussler, Kivinen and Warmuth~\cite{HKW95} showed this order is asymptotically unimprovable in the worst case. The setting above is the finite-horizon game, played for a known number of rounds $T$. Variants change the stopping rule: play until an expert has incurred a fixed number of losses~\cite{AWY08}, or until a horizon that is random and unknown to the player~\cite{LS14}. An important special case of the latter variant is known as the \emph{geometric stopping game}, which ends with probability $\delta$ at each step---a random horizon whose law is known to both players. The geometric stopping game is the main object of study in this paper. 

Beyond the two-expert work of Cover~\cite{Cover65}, exact minimax strategies remained out of reach for fifty years. The first breakthrough was due to Gravin, Peres and Sivan~\cite{GPS16}, who studied the geometric-stopping game with randomized strategies for both players and proved that for three experts the adversary's optimal strategy is COMB: sort the experts by current regret, then advance all odd-ranked experts with probability $\tfrac12$ and all even-ranked experts with probability $\tfrac12$. Their argument relies on an exponential ansatz for the value function that does not extend to $n\ge4$, and they conjectured that COMB is asymptotically optimal for any number of experts, computing the conjectured minimax regret $\pi/(4\sqrt{2\delta})$ in the four-expert case. In the finite-horizon setting, Abbasi-Yadkori, Bartlett and Gabillon~\cite{ABG17} gave near-minimax strategies for three experts, again with a COMB-type adversary, and Gravin, Peres and Sivan~\cite{GPS17} proved tight lower bounds for multiplicative-weights families, showing that no algorithm in that class matches the minimax constant.

A conceptually different route was opened by Drenska and Kohn~\cite{Drenska17,DK20}. Following the approach of Kohn and Serfaty~\cite{KS06,KS10}, who interpreted two-person games as discretizations of nonlinear PDEs, and the tug-of-war games of Peres, Schramm, Sheffield and Wilson~\cite{PSSW09,PS08} for the infinity- and $p$-Laplacians, they treat the dynamic programming principle of the prediction game as a numerical scheme for a nonlinear equation, i.e., ``numerical analysis in reverse''. Using the Barles--Souganidis framework~\cite{BS91} and the viscosity theory of Crandall, Ishii and Lions~\cite{CIL92}, they proved that the rescaled value function for the geometric stopping game converges as $\delta\to0$ to the unique viscosity solution of the degenerate elliptic equation
\begin{equation}\label{eq:pde_general}
  u(x) - \frac12 \max_{\v\in\{0,1\}^n} \v^\top\nabla^2 u(x)\, \v
  = \max_{i} x_i, \qquad x\in\mathbb{R}^n.
\end{equation}
In the finite horizon setting, they proved a similar convergence result as $T\to \infty$ to a related parabolic equation. At a high level, the solution encodes the asymptotically minimax strategies: the player follows expert $i$ with probability $u_{x_i}$, and the adversary's optimal moves are the binary vectors attaining the maximum in~\eqref{eq:pde_general}; however, the connection to the discrete game is subtle and discussed more below. Their analysis covers a general class of symmetric, Lipschitz, translation-invariant payoffs, and for $n=3$ they solved the elliptic equation explicitly (for the max regret payoff), recovering the Gravin--Peres--Sivan result as a continuum limit. The same framework has since been used to obtain potential-based regret bounds for general $n$ by Kobzar, Kohn and Wang~\cite{KKW20a,KKW20b}, to treat history-dependent experts~\cite{CD21,DK22,CD23}, and to analyze limited or malicious adversaries~\cite{BEZ21,BPZ21}. In particular, one of the only existing results for the COMB strategy for any number of experts was obtained by Kobzar, Kohn and Wang \cite{KKW20a,KKW20b}, who showed that COMB is at least as powerful as a randomized adversary for any number of experts. 

The optimality of COMB was settled in the four-expert case by Bayraktar, Ekren and Zhang~\cite{BEZ20}. Working from the Drenska--Kohn PDE, they represented the value function under the conjectured COMB control as a discounted expectation of the local time of an obliquely reflected Brownian motion in an orthant~\cite{Williams95} (from the orthant, one can extend the result to $\R^n$), the local time counting crossings between the two leading experts. Differentiating the dynamic programming principle on the faces of the reflection domain led to a system of first-order hyperbolic equations for the boundary values, which they solved in closed form; a type of a maximum principle for this system then made the verification of the full nonlinear PDE tractable. The result is an explicit formula for $u$, the exact leading-order value $\pi/(4\sqrt{2\delta})$, and the asymptotic optimality of COMB for $n=4$, confirming both conjectures of~\cite{GPS16} in the four-expert case. A companion paper~\cite{BEZ20b} treats the finite-horizon four-expert problem.

For $n\ge5$ experts the picture changes substantially. Numerical experiments by Chase~\cite{Chase19} gave the first indication that COMB is not asymptotically optimal for larger $n$. Calder, Drenska and Mosaphir~\cite{CDM26} then developed numerical methods for the elliptic equation itself: a localization theorem shows the equation can be solved on a large box, with a bound on the effect of boundary errors on a smaller interior box that decays as the box grows. Permutation symmetry and translation invariance reduce the computational domain to the part $x_n=0$ of the ordered sector $x_1\ge\cdots\ge x_n$, making fine-grid solutions feasible up to $n=10$ experts. Their computations support the non-optimality of COMB for $n\ge5$ and single out $(1,0,1,0,0)$, in rank order, as the adversary's optimal direction for five experts. The theory is thus complete for $n\le4$, and $n=5$ is the first case in which the optimal adversarial strategy is expected to depart from COMB.

In this paper we give an explicit solution of the five-expert PDE~\eqref{eq:pde_general}. Our first main result, Theorem~\ref{thm:main}, is that the rank direction $\v_* = (1,0,1,0,0)$, identified numerically in~\cite{CDM26}, indeed attains the Hamiltonian maximum throughout the ordered sector $x_1 \ge \cdots \ge x_5.$ 
The solution is given in three regions of the sector. In the first two regions, it reduces to the four-expert solution of~\cite{BEZ20} plus a single integral of elementary functions against an explicit positive density; in the third, it involves a finite sum of hyperbolic products whose four coefficients are determined by a single scalar quadrature. 
Our formula gives $u(0) = 45\pi^2/(512\sqrt2)$ at the origin, so by the convergence theorem of~\cite{DK20} the minimax regret of the five-expert geometric stopping game started with all experts tied is $u(0)/\sqrt\delta+o(\delta^{-1/2})$ as $\delta\to0$. This is to be compared with the values $\sqrt2/4$, $\sqrt2/3$, and $\pi/(4\sqrt2)$ of $u(0)$ for two, three, and four experts, the first two obtained from the exponential sums \eqref{eq:2expert} and \eqref{eq:3expert}. We also obtain the second-order Taylor expansion of $u$ at the origin, the five-expert analogue of~\cite[equation~(3.4)]{BEZ20}.

The derivation is purely analytic and avoids the stochastic representation used in~\cite{BEZ20}. Solving~\eqref{eq:pde_general} along the characteristic direction $\v_*$ reduces the problem to a one-dimensional two-point boundary value problem, and the permutation-symmetry conditions on the faces of the sector produce coupled systems of first-order hyperbolic PDEs for the boundary traces. The key observation is that these systems can be solved explicitly: a family of truncated hyperbolic modes diagonalizes the trace operators, and the resulting propagators admit a positive Green's function. As a by-product, the same method rederives the four-expert solution of~\cite{BEZ20} in a more direct way, and it explains why the cases $n=2,3$ are elementary while $n\ge4$ are not.

Verifying that the resulting formula is the viscosity solution of \eqref{eq:pde_general} is the bulk of the work in the paper. We prove that $u$ is globally $C^2$ across the region interfaces and the permutation hyperplanes, that $u-\max_i x_i$ is bounded, and that $u$ is the unique viscosity solution in this class. The Hamiltonian inequalities $D^2_\v u \le D^2_{\v_*} u$ for all sixteen non-equivalent binary controls $\v$, where $D^2_\v u=\v^\top\nabla^2u\,\v$, are reduced, by multiaffine interpolation in hyperbolic tangents and by sign principles for Laplace-type integrals, to twenty-one scalar inequalities in one variable. We prove these with $147$ exact rational Bernstein-polynomial certificates; the certificates and independent arithmetic checkers are included in a self-contained, offline-reproducible supplement, and no floating-point or numerical sign sampling enters the proof. The certificates and the reduction to them are also formalized in Lean~4 on top of Mathlib~\cite{LeanCert}, along with the regularity, the boundedness, and the optimality set of COMB; that formalization covers both main theorems apart from the viscosity characterization.

Our second main result, Theorem~\ref{thm:comb}, determines the exact optimality set for the COMB strategy $(1,0,1,0,1)$ in the PDE. In decreasing rank order, COMB attains the Hamiltonian maximum if and only if $x_1 = x_2$ and $x_3 = x_4$; at every other point the curvature gap between $\v_*$ and COMB is strictly positive. The optimality set is thus a three-dimensional set with empty interior in $\mathbb{R}^5$. This confirms and sharpens the numerical results of Calder, Drenska and Mosaphir~\cite{CDM26}. Interestingly, the optimal control is not unique: several other rank-based controls attain the maximum on each region, and we list them in Remark~\ref{rem:additional-controls}. However, up to complementation, $\v_*$ is the only control optimal throughout the entire ordered sector, which also departs from the cases $n=3,4$, where there were always two such strategies.\footnote{We do not actually prove that $\v_*$ is the only such control in this paper, as it is not needed for these results. This is established in our companion paper \cite{calder2026more}.}

The COMB optimality conjecture of Gravin, Peres and Sivan~\cite{GPS16} concerns the game started from the origin $x=0$ where all experts are tied. Since the origin is included in the COMB optimality set $\{x_1=x_2, x_3=x_4\}$ for the continuum PDE, the refutation of the COMB conjecture is not directly implied by the results in this paper, and requires a more careful analysis of the discrete game and the trajectories it can realize (which may drift off of this optimality set). In a forthcoming companion paper \cite{calder2026nofixed} building upon these results, we prove that COMB is not asymptotically minimax optimal for the discrete game with five or more experts.

We were recently made aware of another preprint \cite{bayraktar2026prediction} submitted shortly after ours that establishes the same five-expert solution formula, though with more probabilistic techniques. We refer the reader to \cite{bayraktar2026prediction} for a detailed summary of the differences between the two approaches. 

\section{Main results}\label{sec:main}

We consider the five-expert PDE
\begin{equation}\label{eq:pde}
 u(x)-\frac12\max_{\v\in\{0,1\}^5}\v^\top\nabla^2u(x)\v=\phi(x),
 \qquad \phi(x)=\max_{1\leq i\leq5}x_i,
 \qquad x\in\R^5.
\end{equation}
We work in the ordered sector $\{x_1\geq x_2 \geq x_3 \geq x_4 \geq x_5\}$ and set $k=\sqrt2$. We define the scaled gaps and the correction $F$ by
\begin{equation}\label{eq:scaling}
 y_i=k(x_i-x_{i+1})\geq0,\qquad u(x)=x_1+\frac1kF(y_1,y_2,y_3,y_4).
\end{equation}
Then the ordered sector is mapped to the positive orthant $\{y_i \geq 0\}$ in gap coordinates. The explicit formula for the solution of \eqref{eq:pde} requires different coordinate systems in different regions in the sector. For the leading four experts we set
\begin{equation}\label{eq:four-coordinates}
 z_1=\frac{x_1+x_2-x_3-x_4}{k},\qquad
 z_2=\frac{x_1-x_2+x_3-x_4}{k},\qquad
 z_3=\frac{x_1-x_2-x_3+x_4}{k},
\end{equation}
so that the ordered sector becomes $\{z_1\geq z_2\geq|z_3|\}$, since $k/2=1/k$, $z_1=y_2+\tfrac12(y_1+y_3)$, $z_2=\tfrac12(y_1+y_3)$, and $z_3=\tfrac12(y_1-y_3)$. The fifth expert enters through
\begin{equation}\label{eq:coords12}
 z_4=y_4-\max(z_3,0).
\end{equation}
The sector splits into three regions:
\begin{center}
\begin{tabular}{lll}\toprule
Region & Scaled gaps & Coordinates\\\midrule
I & $y_1\leq y_3$ & $z_3\leq0\leq z_4$\\
II & $y_3\leq y_1\leq y_3+2y_4$ & $z_3\geq0$, $z_4\geq0$\\
III & $y_1\geq y_3+2y_4$ & $z_3\geq 0$, $z_4\leq0$\\\bottomrule
\end{tabular}
\end{center}
The interfaces are $z_3=0$, where $y_1=y_3$, and $z_4=0$ with $z_3\geq0$, where $y_1=y_3+2y_4$. In Region III we instead use the coordinates
\begin{equation}\label{eq:coords3}
 a_1=\frac{y_1-y_3-2y_4}{3},\quad
 a_2=a_1+y_4,\quad a_3=a_2+y_3,\quad a_4=a_3+y_2,
\end{equation}
for which $0\leq a_1\leq a_2\leq a_3\leq a_4$; conversely
\begin{equation}\label{eq:inverse3}
 y_1=a_1+a_2+a_3,\quad y_2=a_4-a_3,\quad
 y_3=a_3-a_2,\quad y_4=a_2-a_1.
\end{equation}
The two systems are related by $a_1=-\tfrac23z_4$ and $a_2=z_3+\tfrac13z_4$, $a_3=z_2+\tfrac13z_4$, $a_4=z_1+\tfrac13z_4$, so they coincide on the interface $z_4=0$, where $a_1=0$ and $(a_4,a_3,a_2)=(z_1,z_2,z_3)$.

The formula is built from the following functions. Write $\theta(L)=\arctan(\eu^{-L})$ and $\lambda(L)=\log\coth(L/2)$. Let
\begin{equation}\label{eq:I-primitive}
\begin{aligned}
 I(X)&=\int_0^X\sinh^4t\cosh^2t\dd=\frac{\sinh6X}{192}-\frac{\sinh4X}{64}
   -\frac{\sinh2X}{64}+\frac X{16}.
\end{aligned}
\end{equation}
The trace $e$ (it is the boundary value $e(X)=F(0,X,0,0)$; see \eqref{eq:e-def}), the only function in the formula that is not elementary, is defined by the quadrature
\begin{equation}\label{eq:e-quadrature}
 e(X)=3\cosh X\int_X^\infty\frac{I(t)}{\cosh^2t\sinh^5t}\dd.
\end{equation}
The integrand is regular at zero, since $I(t)=t^5/5+\O(t^7)$, and decays like $\eu^{-t}/3$ at infinity, so $e$ is positive and bounded. It satisfies the first-order equation
\begin{equation}\label{eq:e-first}
 e'=\tanh X\,e-\frac{3I(X)}{\cosh X\sinh^5X},
\end{equation}
and consequently
\begin{equation}\label{eq:trace-derivatives}
\begin{aligned}
 e''&=-5\coth X\,e'+6e-3\coth X,\\
 e'''&=(31+30\csch^2X)e'-30\coth X\,e+15+18\csch^2X;
\end{aligned}
\end{equation}
these are proved in Section~\ref{sec:slice}. We define the density 
\begin{equation}\label{eq:density}
 p(X)=\frac1{2\sinh^6X}\int_0^X\sinh^6t\,\sech^2t\dd,
\end{equation}
which has the elementary form
\begin{equation}\label{eq:p-elementary}
 p(X)=\tfrac12\tanh X-3\coth X+\frac{15I(X)}{\sinh^6X}.
\end{equation}
Finally, we define the \emph{truncated modes} 
\begin{equation}\label{eq:phi}
 \Phi_t(X)=\frac{\sinh(t-X)}{\sinh t}\quad(0\leq X\leq t),
 \qquad \Phi_t(X)=0\quad(X>t).
\end{equation}

We also need to recall the explicit four-expert solution of Bayraktar, Ekren, and Zhang~\cite[Theorem~3.1]{BEZ20}, which appears as a background term in the five-expert solution. In the present notation, the four-expert solution is given by $u_4=x_1+\tfrac1k F_4$ where
\begin{equation}\label{eq:four-background}
\begin{aligned}
 F_4(z_1,z_2,z_3)={}&\theta(z_1)\cosh z_1\cosh z_2\cosh z_3\\
 &+\tfrac12\lambda(z_1)\sinh z_1\sinh z_2\sinh z_3
  -\tfrac12\sinh(z_2+z_3).
\end{aligned}
\end{equation}
The identification of this formula is given in Section~\ref{sec:region12}. 

We are now ready to state our main result.

\begin{theorem}\label{thm:main}
Let $u(x)=x_1+\tfrac1k F(y)$ on the ordered sector, where in Regions I and II
\begin{equation}\label{eq:formula12}
 F=F_4(z_1,z_2,z_3)
 +\int_{z_1}^\infty\sinh t\,p(t)\eu^{-2z_4\coth t}
            \Phi_t(z_1)\Phi_t(|z_3|)\Phi_t(z_2)\dd,
\end{equation}
and in Region III
\begin{equation}\label{eq:finite3}
\begin{aligned}
 F={}&b_0\cosh a_1\cosh a_2\cosh a_3\\
     &+b_1(\sinh a_1\cosh a_2\cosh a_3+\cosh a_1\sinh a_2\cosh a_3
           +\cosh a_1\cosh a_2\sinh a_3)\\
     &+b_2(\sinh a_1\sinh a_2\cosh a_3+\sinh a_1\cosh a_2\sinh a_3
           +\cosh a_1\sinh a_2\sinh a_3)\\
     &+b_3\sinh a_1\sinh a_2\sinh a_3,
\end{aligned}
\end{equation}
with $e$ and its derivatives evaluated at $a_4$,
\begin{equation}\label{eq:coefficients}
\begin{aligned}
 b_0&=e,\qquad b_1=(e'-3)/6,\\
 b_2&=(e''+6e+18\lambda(a_4)\sinh a_4)/42,\\
 b_3&=(e'''+20e')/336+(6\lambda(a_4)\cosh a_4-13)/14.
\end{aligned}
\end{equation}
At the points with $z_1=0$ in Regions I and II, or $a_4=0$ in Region III, where individual factors such as $\lambda(0)$ are undefined, $F$ is defined as the limit of these formulas from $z_1>0$, respectively $a_4>0$; these limits exist (Section~\ref{sec:reg}). Extend $u$ to $\R^5$ by sorting the coordinates. Then $u\in C^2(\R^5)$, $u-\phi$ is bounded, and $u$ is the unique viscosity solution of \eqref{eq:pde} in this class. In the ordered sector, the rank direction $\v_*=(1,0,1,0,0)$ attains the Hamiltonian maximum everywhere and
\begin{equation}\label{eq:origin-expansion}
 u(x)=\frac{45\pi^2}{512\sqrt2}+\frac15\sum_{i=1}^5x_i
 +\frac{15\pi^2}{256\sqrt2}\Bigl(\sum_{i=1}^5x_i^2-\frac12\sum_{i<j}x_ix_j\Bigr)
 +o(|x|^2) \ \ \text{as}  \ \ x\to 0.
\end{equation}
\end{theorem}

The two expressions for $F$ agree on the interface $z_4=0$ (Section~\ref{sec:region3}). In Regions I and II the formula requires one quadrature of elementary functions; in Region III only the value $e(a_4)$ is not elementary, since \eqref{eq:e-first} and \eqref{eq:trace-derivatives} eliminate the derivatives in \eqref{eq:coefficients}. The constant term $u(0)=45\pi^2/(512\sqrt2)$ in \eqref{eq:origin-expansion} is obtained in Lemma~\ref{lem:e0} by evaluating the quadrature \eqref{eq:e-quadrature} at $X=0$; it is the five-expert counterpart of the four-expert value $u_4(0)=\pi/(4\sqrt2)$ of \cite{BEZ20}. The expansion \eqref{eq:origin-expansion} is the five-expert analogue of \cite[equation~(3.4)]{BEZ20} and is proved in Section~\ref{sec:verification} from the $C^2$ regularity, the symmetries, and the equation at the origin.

Our second result identifies the set on which the COMB control is optimal. In decreasing rank order, COMB is $\v_C=(1,0,1,0,1)$. Its complement $\one-\v_C=(0,1,0,1,0)$ has the same curvature, because $\nabla^2u\one=0$. Writing $D_\v^2u=\v^\top\nabla^2u\,\v$ for the second derivative of $u$ in the direction $\v$, define the physical curvature gap
\begin{equation}\label{eq:comb-gap}
 \Delta(x)=D_{\v_*}^2u(x)-D_{\v_C}^2u(x).
\end{equation}
Since $\v_*$ attains the Hamiltonian maximum by Theorem~\ref{thm:main}, $\Delta\geq0$ everywhere in the ordered sector. The following theorem identifies the set where equality holds.

\begin{theorem}\label{thm:comb}
For $x_1\geq\cdots\geq x_5$, COMB attains the Hamiltonian maximum, i.e., $\Delta(x)=0$, if and only if
\begin{equation}\label{eq:comb-set}
x_1=x_2\quad\text{and}\quad x_3=x_4.
\end{equation}
At every other point $\Delta>0$. 
\end{theorem}

We write a binary control $\v$ as the string $v_1v_2v_3v_4v_5$, e.g.\ $\v_*=10100$, and identify it with its complement $\one-\v$, since $\nabla^2u\one=0$ gives $D^2_{\one-\v}u=D^2_\v u$; Remark~\ref{rem:additional-controls} lists the representatives with $v_1=1$, and the tables of Section~\ref{sec:verification} use those with $v_1=0$.

\begin{remark}\label{rem:additional-controls}
The optimal control is not unique even away from COMB's equality set. In addition to $\v_*=(1,0,1,0,0)$, the following controls are optimal throughout the indicated regions:
\begin{equation}\label{eq:additional-controls}
 \mathrm I:\ 10011,\qquad
 \mathrm{II}:\ 10010,\qquad
 \mathrm{III}:\ 10001\ \hbox{and}\ 10010.
\end{equation}
Their directions are $\partial_{z_3}$ in Regions I and II, and $\partial_{a_1}$ and $\partial_{a_2}$ in Region III.
This lists additional optimizers; further ties may occur on boundaries.
\end{remark}

The rest of the paper is organized into two sections; in Section~\ref{sec:derivation} we derive the formula and in Section~\ref{sec:verification} we verify the Hamiltonian inequalities that establish the formula indeed gives the solution of \eqref{eq:pde}. Theorem \ref{thm:comb} is an immediate consequence of this verification work. The verification is a tedious process, and uses a computer assisted proof to do much of the heavy lifting. This is described in detail in Section \ref{sec:verification} and in the supplement.

\section{Derivation of the five-expert formula}\label{sec:derivation}

This section derives the formula of Theorem~\ref{thm:main}, following the
steps outlined below. The proof that the result solves \eqref{eq:pde} is given in Section~\ref{sec:verification}.

Before summarizing our approach, we make some observations about \eqref{eq:pde_general}. Since the payoff $\phi(x)=\max_ix_i$ satisfies $\phi(x+c\one)=\phi(x)+c$ for all $c\in\R$ and is permutation invariant in the coordinates, the same is true for the viscosity solution $u$ of \eqref{eq:pde_general}, by uniqueness. Wherever $u$ is differentiable it follows that $\nabla u\cdot \one = 1$ and $u_{x_i}=u_{x_j}$ whenever $x_i=x_j$. In particular, this leads to an origin boundary condition 
\begin{equation}\label{eq:origin}
\nabla u(0) = \frac{1}{n}\one
\end{equation}
for the general $n$ expert problem. 

As a warmup and to illustrate the main ideas in this section, let us recall how to derive the solution formula for \eqref{eq:pde_general} in the case of $n=2$ and $n=3$ experts. For $n=2$, define $f(x_1) = u(x_1,0)$. Since $\v =(1,0)$ is optimal for 2 experts \eqref{eq:pde_general} reduces to 
\[f - \frac{1}{2}f''(x_1) = x_1 \ \ \text{ for } x_1 \geq 0.\]
By \eqref{eq:origin} we have $f'(0) = \frac{1}{2}$, and since $u$ has at most linear growth, we can ignore the exponentially growing solution $\eu^{\sqrt 2 x_1}$. Thus
\[f(x_1) = x_1 + \frac{1}{2\sqrt 2}\eu^{-\sqrt 2 x_1}.\]
Since $\nabla u\cdot \one = 1$ we arrive at the $n=2$ expert formula
\begin{equation}\label{eq:2expert}
u(x) = x_2 + f(x_1-x_2) = x_1 + \frac{1}{2\sqrt 2}\eu^{\sqrt 2(x_2 - x_1)}.
\end{equation}
A very similar argument for $n=3$ with optimal strategy $\v=(1,0,0)$ yields 
\begin{equation}\label{eq:3expert}
u(x) = x_1 + \frac{1}{2\sqrt 2}\eu^{\sqrt 2(x_2 - x_1)} + \frac{1}{6\sqrt 2}\eu^{\sqrt 2(2x_3 - x_2 - x_1)}.
\end{equation}
The key fact making both of these cases trivial is that there exists an optimal strategy, $\v=(1,0)$ for $n=2$ and $\v=(1,0,0)$ for $n=3$, whose characteristic direction shoots off to infinity in one direction, allowing us to neglect the exponentially growing mode. This fails for the $n=4$ and $n=5$ expert problems. For $n=4$ two strategies are optimal $\v = (1,0,1,0)$ and $\v=(1,0,0,1)$ \cite{BEZ20}, and neither has this property. Both characteristic directions intersect the boundary of the sector $\{x_1\geq x_2 \geq x_3 \geq x_4\}$ in two places, requiring the solution of a two point boundary value problem, which greatly complicates the solution derivation. The case of $n=5$ is more complicated, since the characteristic direction of the strategy $\v_* = (1,0,1,0,0)$ that we follow, which is optimal throughout the sector but not the only optimal strategy (see Remark~\ref{rem:additional-controls}), hits different faces of the sector boundary depending on the location, necessitating splitting the sector into regions and defining the solution differently in each region. 

Nevertheless, the overall idea from $n=2$ and $n=3$ experts---solving a one dimensional ODE in the optimal direction and resolving constants through boundary conditions---carries over to $n=4$ and $n=5$ at a high level, and is carried out in the rest of this section.  In Section \ref{sec:characteristics} we solve the equation \eqref{eq:pde} along the direction $\v_* = (1,0,1,0,0)$, which was conjectured as optimal via numerics in \cite{CDM26}. The boundary conditions on the faces $x_i=x_{i+1}$ resulting from permutation of coordinates leads to a coupled system of first order PDEs, which are derived in Section \ref{sec:traces}. The key insight in this section is that we can (more or less) explicitly solve this system of PDEs; we show how to do this in Section \ref{sec:modes}. Then the final two subsections establish the five-expert solution formula in the respective regions. Along the way, we give a rederivation of the $n=4$ expert solution, avoiding the stochastic techniques used in \cite{BEZ20} (see Section \ref{sec:region12}).

\subsection{Solving the ODE along the \texorpdfstring{$\v_*$}{v*} characteristic}\label{sec:characteristics}

We work in the ordered sector with the scaled gaps and the correction $F$
of \eqref{eq:scaling}. For a binary control, we define $b_\v=(v_1-v_2,v_2-v_3,v_3-v_4,v_4-v_5)^\top$.
Since $y_i=k(x_i-x_{i+1})$, a displacement of $x$ in the direction $\v$
moves $y$ in the direction $kb_\v$, so
$\v^\top\nabla^2u\,\v=k\,b_\v^\top\nabla_y^2F\,b_\v$. Hence, with $D_\v^2u$ as in Section~\ref{sec:main}
and using $k=2/k$,
\begin{equation}\label{eq:direction-scaling}
 D_\v^2u=\frac2k\,b_\v^\top\nabla_y^2F b_\v.
\end{equation}
We assume that $\v_*$ attains the maximum in \eqref{eq:pde}, which
Section~\ref{sec:verification} verifies a posteriori. In the sector,
where $\phi=x_1$, the equation then reads $u-\tfrac12D_{\v_*}^2u=x_1$,
and with $u=x_1+F/k$ and $b_{\v_*}=(1,-1,1,0)$ this is, with $\partial_i=\partial_{y_i}$, the fixed equation
\begin{equation}\label{eq:fixed}
 (\partial_1-\partial_2+\partial_3)^2F=F.
\end{equation}
Since $u_{x_1}=1+F_{y_1}$, $u_{x_2}=-F_{y_1}+F_{y_2}$, $u_{x_3}=-F_{y_2}+F_{y_3}$, $u_{x_4}=-F_{y_3}+F_{y_4}$, and $u_{x_5}=-F_{y_4}$ the face conditions $u_{x_i}=u_{x_{i+1}}$ when $x_i=x_{i+1}$ (or $y_i=0$) produce
\begin{equation}\label{eq:faces}
\begin{aligned}
 2F_{y_1}-F_{y_2}&=-1 &&(y_1=0),&-F_{y_1}+2F_{y_2}-F_{y_3}&=0 &&(y_2=0),\\
 -F_{y_2}+2F_{y_3}-F_{y_4}&=0 &&(y_3=0),&-F_{y_3}+2F_{y_4}&=0 &&(y_4=0).
\end{aligned}
\end{equation}
We now change variables so one direction aligns with $\v_*$. Let $\gamma=y_1+y_2$, $\delta=y_2+y_3$ and $s=y_2$. In these new coordinates write $G(s,\gamma,\delta,y_4)=F(\gamma-s,s,\delta-s,y_4)$ so that $(\partial_1-\partial_2+\partial_3)F=-G_s$. The sector becomes
$0\leq s\leq\min(\gamma,\delta)$, and \eqref{eq:fixed} turns into $G_{ss}=G$.
For fixed $(\gamma,\delta,y_4)$, this is a second-order linear equation in $s$ on the
interval $[0,M]$, where $M=\min(\gamma,\delta)$. The solution to this two point boundary value problem is given by 
\[
 G(s,\gamma,\delta,y_4)=\frac{\sinh(M-s)}{\sinh M}\,G(0,\gamma,\delta,y_4)
           +\frac{\sinh s}{\sinh M}\,G(M,\gamma,\delta,y_4)
 \qquad(0\leq s\leq M).
\]
The endpoint $s=0$ lies on the face $y_2=0$. The endpoint $s=M$ lies on
the face $y_1=0$ when $\gamma\leq \delta$, that is, when $y_1\leq y_3$, and on the
face $y_3=0$ when $y_3\leq y_1$. To express the endpoint values, we
introduce the four traces
\begin{equation}\label{eq:trace-defs}
\begin{aligned}
 a(X,Y,y_4)&=F(X,0,X+Y,y_4),&
 \ell(X,Y,y_4)&=F(0,X,Y,y_4),\\
 b(X,Y,y_4)&=F(X+Y,0,X,y_4),&
 q(X,Y,y_4)&=F(Y,X,0,y_4).
\end{aligned}
\end{equation}
Here $a$ and $b$ parametrize the face $y_2=0$ on the two sides of the
diagonal $y_1=y_3$, while $\ell$ and $q$ parametrize the faces $y_1=0$ and
$y_3=0$. For $y_1\leq y_3$, set $X=y_1+y_2$ and $Y=y_3-y_1$. Then $Y\geq 0,$ so $M=\min(\gamma,\delta)=\gamma,$ so $M=\gamma=X$
and $\delta=X+Y$, so the characteristic through $y$, by its construction, has endpoints
$(X,0,X+Y,y_4)$ at $s=0$ and $(0,X,Y,y_4)$ at $s=X$, where $F$ takes the
values $a(X,Y,y_4)$ and $\ell(X,Y,y_4)$. Since $M-s=y_1$ and $s=y_2$, the
interpolation formula becomes
\begin{equation}\label{eq:reconstruct-left}
 F(y)=\frac{\sinh y_1}{\sinh X}a(X,Y,y_4)
       +\frac{\sinh y_2}{\sinh X}\ell(X,Y,y_4).
\end{equation}
For $y_3\leq y_1$, set $X=y_2+y_3$ and $Y=y_1-y_3$. Then $Y\geq 0,$ so $M=\min(\gamma,\delta)=\delta,$ so $M=\delta=X$ and
$\gamma=X+Y$, so the endpoints are $(X+Y,0,X,y_4)$ and $(Y,X,0,y_4)$, where $F$
takes the values $b(X,Y,y_4)$ and $q(X,Y,y_4)$. Now $M-s=y_3$, and we obtain
\begin{equation}\label{eq:reconstruct-right}
 F(y)=\frac{\sinh y_3}{\sinh X}b(X,Y,y_4)
       +\frac{\sinh y_2}{\sinh X}q(X,Y,y_4).
\end{equation}
The quantities $a$, $b$, $\ell$, and $q$ satisfy coupled PDEs arising from the face conditions \eqref{eq:faces}, which is the subject of the next section.

\subsection{Deriving the trace systems}\label{sec:traces}

Substituting \eqref{eq:reconstruct-left} into the first two face conditions gives
\begin{equation}\label{eq:left-system}
 a_X=2\coth X\,a-2\csch X\,\ell,
 \qquad \ell_X-2\ell_Y=2\coth X\,\ell-2\csch X\,a-1.
\end{equation}
To see this, write \eqref{eq:reconstruct-left} as $F=w_1a+w_2\ell$ with
\[
 w_1=\frac{\sinh y_1}{\sinh X},\qquad w_2=\frac{\sinh y_2}{\sinh X},
\]
where $a$ and $\ell$ are evaluated at $(X,Y,y_4)$ with $(X,Y)=(y_1+y_2,\,y_3-y_1)$.
Since $X_{y_1}=X_{y_2}=1$, $Y_{y_1}=-1$, and $Y_{y_3}=1$,
the chain rule gives, for $g\in\{a,\ell\}$,
\[
 \partial_1g=g_X-g_Y,\qquad \partial_2g=g_X,\qquad
 \partial_3g=g_Y,\qquad \partial_4g=g_{y_4}.
\]
The weights depend only on $y_1$ and $y_2$, through the numerator and
through $X$, and
\[
\begin{aligned}
 \partial_1w_1&=\frac{\cosh y_1}{\sinh X}-\coth X\,w_1,&
 \partial_2w_1&=-\coth X\,w_1,\\
 \partial_1w_2&=-\coth X\,w_2,&
 \partial_2w_2&=\frac{\cosh y_2}{\sinh X}-\coth X\,w_2.
\end{aligned}
\]
Therefore
\[
\begin{aligned}
 F_{y_1}&=(\partial_1w_1)a+(\partial_1w_2)\ell
       +w_1(a_X-a_Y)+w_2(\ell_X-\ell_Y),\\
 F_{y_2}&=(\partial_2w_1)a+(\partial_2w_2)\ell+w_1a_X+w_2\ell_X,\\
 F_{y_3}&=w_1a_Y+w_2\ell_Y,\qquad F_{y_4}=w_1a_{y_4}+w_2\ell_{y_4}.
\end{aligned}
\]
On the face $y_2=0$ we have $X=y_1$, so $w_1=1$, $w_2=0$,
$\partial_1w_1=\partial_1w_2=0$, $\partial_2w_1=-\coth X$,
and $\partial_2w_2=\csch X$. Thus
\[
 F_{y_1}=a_X-a_Y,\qquad F_{y_2}=a_X-\coth X\,a+\csch X\,\ell,\qquad
 F_{y_3}=a_Y,\qquad F_{y_4}=a_{y_4}.
\]
The second face condition in \eqref{eq:faces} reads $2F_{y_2}=F_{y_1}+F_{y_3}=a_X$,
which is the first identity in \eqref{eq:left-system}.
On the face $y_1=0$ we have $X=y_2$, so $w_1=0$, $w_2=1$,
$\partial_2w_1=\partial_2w_2=0$, $\partial_1w_1=\csch X$,
and $\partial_1w_2=-\coth X$. Thus
\[
 F_{y_1}=\ell_X-\ell_Y-\coth X\,\ell+\csch X\,a,\qquad F_{y_2}=\ell_X,\qquad
 F_{y_3}=\ell_Y,\qquad F_{y_4}=\ell_{y_4}.
\]
The first face condition reads $2F_{y_1}-F_{y_2}=-1$, which is the second
identity in \eqref{eq:left-system}.

The same calculation applies to \eqref{eq:reconstruct-right}, where
$X=y_2+y_3$, $Y=y_1-y_3$, the weights are $\sinh y_3/\sinh X$ and
$\sinh y_2/\sinh X$, and for $g\in\{b,q\}$ the chain rule gives
$\partial_1g=g_Y$, $\partial_2g=g_X$, $\partial_3g=g_X-g_Y$, and
$\partial_4g=g_{y_4}$. On the face $y_2=0$ this yields
\[
 F_{y_1}=b_Y,\qquad F_{y_2}=b_X-\coth X\,b+\csch X\,q,\qquad
 F_{y_3}=b_X-b_Y,\qquad F_{y_4}=b_{y_4},
\]
and on the face $y_3=0$
\[
 F_{y_1}=q_Y,\qquad F_{y_2}=q_X,\qquad
 F_{y_3}=q_X-q_Y-\coth X\,q+\csch X\,b,\qquad F_{y_4}=q_{y_4}.
\]
The second face condition, $2F_{y_2}=F_{y_1}+F_{y_3}$, and the third face condition,
$2F_{y_3}=F_{y_2}+F_{y_4}$, give
\begin{equation}\label{eq:right-system}
 b_X=2\coth X\,b-2\csch X\,q,
 \qquad q_X-2q_Y-q_{y_4}=2\coth X\,q-2\csch X\,b.
\end{equation}
The fourth face condition, $F_{y_3}=2F_{y_4}$, becomes
\begin{equation}\label{eq:bottom-traces}
 a_Y=2a_{y_4},\quad \ell_Y=2\ell_{y_4},\quad
 b_X-b_Y=2b_{y_4},\quad q_X=3q_{y_4}\qquad(y_4=0).
\end{equation}
The first three follow directly from the face values of $F_{y_3}$ and $F_{y_4}$
displayed above. At the $q$ trace the third face says $2F_{y_3}=q_X+q_{y_4}$,
while the bottom face says $F_{y_3}=2q_{y_4}$.

Our main task for much of this section will be to solve the coupled system of PDEs \eqref{eq:left-system}. A first step in this direction is to note that both equations can be integrated. Indeed, we note that
\[\frac{\diff}{\diff X}\left(\frac{a}{\sinh^2X}\right)=\frac{a_X - 2\coth X\,a}{\sinh^2X} = -\frac{2\ell}{\sinh^3 X}\]
and then integrate, using boundedness at infinity, to obtain $a=\Iop \ell$, where the integral operator $\Iop$ is defined by 
\begin{equation}\label{eq:operators}
(\Iop f)(X)=2\sinh^2X\int_X^\infty\frac{f(t)}{\sinh^3t}\dd,
\end{equation}
for $X>0$ and by continuity for $X=0$. A similar argument yields $b=\Iop q$. We also define, for later usage, the operators
\begin{equation}\label{eq:kop}
\Kop=-\coth X+\csch X\,\Iop
\end{equation}
and 
\begin{equation}\label{eq:generators}
\Lop_m=\frac1m\partial_X+\Kop
\end{equation}
for $m=2,3$. 
The remaining trace equations in \eqref{eq:left-system} and
\eqref{eq:right-system} can now be written as
\begin{equation}\label{eq:trace-Y}
 \ell_Y=\Lop_2\ell+\tfrac12,\qquad q_Y+\tfrac12q_{y_4}=\Lop_2q.
\end{equation}

\subsection{The closed slice and its single scalar quadrature}\label{sec:slice}

Before solving the trace equations in general, we first note that we can solve them on the slice $y_1=y_3$, $y_4=0$,
since the equations close on this set. We define
$\Psi(y_1,y_2)=F(y_1,y_2,y_1,0)$. Equations \eqref{eq:fixed}--\eqref{eq:faces} imply
\begin{equation}\label{eq:slice}
 (\partial_{y_1}-\partial_{y_2})^2\Psi=\Psi,\qquad
 \Psi_{y_1}-\tfrac76\Psi_{y_2}=-\tfrac12\ (y_1=0),\qquad
 \Psi_{y_2}-\tfrac12\Psi_{y_1}=0\ (y_2=0).
\end{equation}
The direction $(1,-1,1,0)$ of \eqref{eq:fixed} is tangent to the slice
and corresponds to $\partial_{y_1}-\partial_{y_2}$ there, since
$\Psi_{y_1}=F_{y_1}+F_{y_3}$ and $\Psi_{y_2}=F_{y_2}$; this gives the first
equation. At $y_1=0$ the faces $y_1=0$, $y_3=0$, and $y_4=0$ all apply, and
\eqref{eq:faces} gives $F_{y_1}=(F_{y_2}-1)/2$, $F_{y_4}=F_{y_3}/2$, and
hence $F_{y_3}=2F_{y_2}/3$, so $\Psi_{y_1}=7\Psi_{y_2}/6-1/2$. At $y_2=0$ the
second condition in \eqref{eq:faces} gives
$\Psi_{y_2}=F_{y_2}=\tfrac12(F_{y_1}+F_{y_3})=\tfrac12\Psi_{y_1}$. Note that $X=y_1+y_2$ is the variable of the traces \eqref{eq:trace-defs}:
the expressions $y_1+y_2$ and $y_2+y_3$ coincide on the slice, and the
endpoint values $\Psi(X,0)=F(X,0,X,0)$ and $\Psi(0,X)=F(0,X,0,0)$ are the traces
$a=b$ and $\ell=q$ at $Y=y_4=0$. 

Let $h(X)=\Psi(X,0)$.
The direction $(1,-1)$ leaves $X$ invariant. Along the characteristic through
$(y_1,y_2)$, parametrized by $s=y_2$ with $X$ fixed, the first equation in
\eqref{eq:slice} reads $\partial_s^2\Psi(X-s,s)=\Psi(X-s,s)$, so
\[
 \Psi(X-s,s)=h(X)\cosh s+c(X)\sinh s\qquad(0\leq s\leq X)
\]
for some $c$, and differentiating in $s$ at $s=0$ gives
$c=-\Psi_{y_1}(X,0)+\Psi_{y_2}(X,0)$. Here $\Psi_{y_1}(X,0)=h'(X)$, since $h$ is the
restriction of $\Psi$ to the face $y_2=0$, and the boundary condition at $y_2=0$
gives $\Psi_{y_2}(X,0)=\tfrac12h'(X)$. Thus $c=-\tfrac12h'$, and
\begin{equation}\label{eq:slice-reconstruction}
 \Psi(y_1,y_2)=h(X)\cosh y_2-\tfrac12h'(X)\sinh y_2.
\end{equation}
This is the hyperbolic interpolation \eqref{eq:reconstruct-left} on the slice,
after the boundary condition at $y_2=0$ has been used to express the endpoint
value $\Psi(0,X)$ through $h$ and $h'$.
The boundary condition at $y_1=0$ now determines $h$. Since
\eqref{eq:slice-reconstruction} depends on $y_1$ only through $X$,
\[
 \Psi_{y_1}=h'\cosh y_2-\tfrac12h''\sinh y_2,\qquad
 \Psi_{y_2}=\Psi_{y_1}+h\sinh y_2-\tfrac12h'\cosh y_2.
\]
Substituting into $\Psi_{y_1}-\tfrac76\Psi_{y_2}=-\tfrac12$ at $y_1=0$, where $y_2=X$,
gives
\[
 \tfrac1{12}h''\sinh X+\tfrac5{12}h'\cosh X-\tfrac76h\sinh X=-\tfrac12,
\]
and multiplying by $12/\sinh X$ gives
\begin{equation}\label{eq:alpha-ode}
 h''+5\coth X\,h'-14h=-6\csch X.
\end{equation}
It is convenient to obtain the endpoint value $e(X)=\Psi(0,X)$ directly
from this equation; we show that it is the function \eqref{eq:e-quadrature}
of Section~\ref{sec:main}. By \eqref{eq:slice-reconstruction},
\begin{equation}\label{eq:e-def}
 e(X)=\Psi(0,X)=h(X)\cosh X-\tfrac12h'(X)\sinh X.
\end{equation}
Differentiating and eliminating $h''$ by \eqref{eq:alpha-ode} gives
\begin{equation}\label{eq:e-prime}
 e'=3h'\cosh X-6h\sinh X+3,
\end{equation}
and differentiating once more,
$e''=3h''\cosh X-3h'\sinh X-6h\cosh X$. Substituting $h''$ from
\eqref{eq:alpha-ode} into $e''+5\coth X\,e'-6e$, the terms in $h'$ and
in $h$ cancel and the constants leave
\begin{equation}\label{eq:e-ode}
 e''+5\coth X\,e'-6e=-3\coth X.
\end{equation}
Since \eqref{eq:e-def} rearranges to $h'=2\coth X\,h-2\csch X\,e$
and $h$ is bounded, the argument after \eqref{eq:operators} gives
$h=\Iop e$. These identities also give the compatibility relation
\begin{equation}\label{eq:compat}
 e'=6\Kop e+3.
\end{equation}
Indeed, \eqref{eq:e-prime} gives $e'=3h'\cosh X-6h\sinh X+3$, while
$\Kop e=h'\cosh X/2-h\sinh X$ by $h=\Iop e$ and \eqref{eq:e-def}.

Conversely, let $e$ be a solution of \eqref{eq:e-ode} with $e$ and $e'$
bounded, and put $h=\Iop e$. Then \eqref{eq:e-def} holds, since
$\Iop e$ solves $h'=2\coth X\,h-2\csch X\,e$. Moreover, the
function $\tilde h=\tfrac16\sinh X\,e'+\cosh X\,e-\tfrac12\sinh X$
satisfies
\[
 \tilde h'-2\coth X\,\tilde h+2\csch X\,e
 =\tfrac16\sinh X\bigl(e''+5\coth X\,e'-6e+3\coth X\bigr)=0,
\]
so $\tilde h-\Iop e=C\sinh^2X$ for a constant $C$. Since
$\tilde h-\Iop e=\O(\eu^X)$, $C=0$, and $\Iop e=\tilde h$ is
\eqref{eq:compat}. Differentiating $h'=2\coth X\,h-2\csch X\,e$
and eliminating $e'$ and $e$ by \eqref{eq:compat} and \eqref{eq:e-def}
gives \eqref{eq:alpha-ode}. Thus it suffices to solve \eqref{eq:e-ode}.

To reduce \eqref{eq:e-ode} to a first-order equation, we set
$\eta=e-e''$. Differentiating \eqref{eq:e-ode} gives
$e'''=-5\coth X\,e''+(5\csch^2X+6)e'+3\csch^2X$, so that
$\eta'=e'-e'''$ and $6\coth X\,\eta$ combine to
\[
 \eta'+6\coth X\,\eta
 =-\coth X\,(e''+5\coth X\,e'-6e)-3\csch^2X
 =3\coth^2X-3\csch^2X=3.
\]
Thus
\begin{equation}\label{eq:eta}
 \eta'+6\coth X\,\eta=3,
 \qquad \eta(X)=\frac3{\sinh^6X}\int_0^X\sinh^6t\dd.
\end{equation}
To integrate, multiply by $\sinh^6X$, whose logarithmic derivative is
$6\coth X$, and integrate from zero:
\[
 (\sinh^6X\,\eta)'=3\sinh^6X,\qquad
 \eta(X)=\frac3{\sinh^6X}\int_0^X\sinh^6t\dd+\frac{C}{\sinh^6X}.
\]
Since the trace $e$ is twice differentiable at $X=0$, $\eta=e-e''$ is
bounded there, whereas $C/\sinh^6X$ is not unless $C=0$. This gives the
displayed formula, with $\eta(X)\sim3X/7$ at zero and
$\eta(X)\to\tfrac12$ at infinity.
Since $e''=e-\eta$, equation \eqref{eq:e-ode} now gives
$e'=\tanh X(e+\eta/5)-3/5$. With the primitive $I$ of
\eqref{eq:I-primitive}, integrating
$(\sinh^5X\cosh X)'=5\sinh^4X+6\sinh^6X$ from zero and using
$I=\int_0^X\sinh^4t\dd+\int_0^X\sinh^6t\dd$ give
\[
 \int_0^X\sinh^6t\dd=\sinh^5X\cosh X-5I(X),\qquad
 \eta=3\coth X-\frac{15I(X)}{\sinh^6X}.
\]
Substituting this into $e'=\tanh X(e+\eta/5)-3/5$, the terms
$\tfrac35$ cancel and the trace equation becomes the first-order equation
\eqref{eq:e-first}. With the integrating factor $\sech X$ it reads
$(\sech X\,e)'=-3I(X)/(\cosh^2X\sinh^5X)$, and since $\sech X\,e\to0$
at infinity for bounded $e$, integration from $X$ to infinity yields the
quadrature \eqref{eq:e-quadrature}. The integrand there is regular at zero
since $I(t)=t^5/5+\O(t^7)$, and decays exponentially at infinity. Thus $e$
is positive and bounded. Equation \eqref{eq:e-first} gives
\begin{equation}\label{eq:e-endpoints}
 e'(0)=-\tfrac35,\qquad e''(0)=e(0),\qquad
 e(X)\longrightarrow\tfrac12\quad(X\to\infty).
\end{equation}
Indeed, since $\sinh^4t\cosh^2t=t^4+\tfrac53t^6+\O(t^8)$, we have
$I(X)=\tfrac15X^5+\tfrac5{21}X^7+\O(X^9)$ and
$\cosh X\sinh^5X=X^5+\tfrac43X^7+\O(X^9)$, so the forcing term in
\eqref{eq:e-first} is
\[
 \frac{3I(X)}{\cosh X\sinh^5X}=\tfrac35-\tfrac3{35}X^2+\O(X^4).
\]
As $\tanh X\,e\to0$ at zero, \eqref{eq:e-first} gives $e'(0)=-\tfrac35$.
The identity $e''(0)=e(0)$ is $\eta(0)=0$, which is immediate from
\eqref{eq:eta}; it also follows by differentiating \eqref{eq:e-first},
since the forcing term is even. At infinity, $I(t)\sim \eu^{6t}/384$ and
$\cosh^2t\sinh^5t\sim \eu^{7t}/128$, so the integrand in
\eqref{eq:e-quadrature} is asymptotic to $\eu^{-t}/3$, and
$e(X)\sim3\cdot\tfrac12\eu^{X}\cdot\tfrac13\eu^{-X}=\tfrac12$.
The forcing term in \eqref{eq:e-first} has a removable singularity at the
origin and extends analytically there. Thus the trace is analytic at the origin. At infinity, substituting a series in
$\eu^{-X}$ into \eqref{eq:e-first} gives
\begin{equation}\label{eq:e-asymptotic}
 e(X)=\tfrac12+\tfrac12\eu^{-2X}-\tfrac25\eu^{-4X}+\O(X\eu^{-6X}).
\end{equation}
Indeed, the forcing term $3I/(\cosh X\sinh^5X)$ is a rational function
of $\eu^{-2X}$ apart from the term $X/16$ of $I$, which first contributes at
order $X\eu^{-6X}$, and the coefficients of $\eu^{-2X}$ and $\eu^{-4X}$ are
determined by matching; the remainder satisfies a linear equation with
forcing $\O(X\eu^{-6X})$, whose bounded solution is $\O(X\eu^{-6X})$.
Finally we evaluate $e(0)$ in closed form.
\begin{lemma}\label{lem:e0}
$e(0)=45\pi^2/512$; hence $u(0)=e(0)/\sqrt2=45\pi^2/(512\sqrt2)$, the
constant term in \eqref{eq:origin-expansion}.
\end{lemma}
\begin{proof}
By \eqref{eq:e-quadrature}, $e(0)=3J$ with
$J=\int_0^\infty I(t)w(t)\dd$ and $w=\sech^2t\,\csch^5t$. Since
$I(t)=\int_0^t\sinh^4s\cosh^2s\,\diff s$ by \eqref{eq:I-primitive} and all
integrands are positive, Fubini's theorem gives
\[
 J=\int_0^\infty\sinh^4s\cosh^2s\,V(s)\,\diff s,\qquad
 V(s)=\int_s^\infty w(t)\dd.
\]
Writing $1=\cosh^2t-\sinh^2t$ three times,
$w=\csch^5t-\csch^3t+\csch t-\sinh t\,\sech^2t$, and the standard
antiderivatives, with $\lambda'=-\csch$, give
\[
 V(s)=\tfrac14\csch^3s\coth s-\tfrac78\csch s\coth s
      +\tfrac{15}8\lambda(s)-\sech s,
\]
each term vanishing at infinity. Hence
\[
 \sinh^4s\cosh^2s\,V(s)=\tfrac14\cosh^3s-\tfrac78\sinh^2s\cosh^3s
 -\sinh^4s\cosh s+\tfrac{15}8\lambda(s)\sinh^4s\cosh^2s.
\]
We integrate over $[0,\tau]$ and let $\tau\to\infty$. The first three terms are
elementary. In the fourth, $\sinh^4s\cosh^2s=I'(s)$ and $\lambda'=-\csch$,
so integrating by parts, with $\lambda(s)I(s)\to0$ at $s=0$,
\[
 \int_0^\tau\lambda(s)I'(s)\,\diff s=\lambda(\tau)I(\tau)+\int_0^\tau\csch s\,I(s)\,\diff s,
\]
and since $\sinh2ms/\sinh s=2\sum_{j=1}^m\cosh(2j-1)s$,
\[
 \csch s\,I(s)=\tfrac1{96}\cosh5s-\tfrac1{48}\cosh3s-\tfrac5{96}\cosh s
 +\frac{s}{16\sinh s}.
\]
Collecting the elementary terms,
\[
\begin{aligned}
 \int_0^\tau\sinh^4s\cosh^2s\,V(s)\,\diff s
 ={}&E(\tau)+\frac{15}{128}\int_0^\tau\frac{s\,\diff s}{\sinh s},\\
 E(\tau)={}&\tfrac14\bigl(\sinh\tau+\tfrac13\sinh^3\tau\bigr)
   -\tfrac78\bigl(\tfrac13\sinh^3\tau+\tfrac15\sinh^5\tau\bigr)
   -\tfrac15\sinh^5\tau\\
 &+\tfrac{15}8\Bigl(\lambda(\tau)I(\tau)+\tfrac1{480}\sinh5\tau
   -\tfrac1{144}\sinh3\tau-\tfrac5{96}\sinh\tau\Bigr).
\end{aligned}
\]
Expanding in $\eu^{-\tau}$, with $\lambda(\tau)=2\sum_{j\ \mathrm{odd}}\eu^{-j\tau}/j$,
the growing exponentials cancel and
$E(\tau)=\bigl(\tfrac{15}{64}\tau-\tfrac{23}{448}\bigr)\eu^{-\tau}+\O(\tau \eu^{-3\tau})\to0$.
Therefore
\[
 J=\frac{15}{128}\int_0^\infty\frac{s\,\diff s}{\sinh s}
  =\frac{15}{128}\cdot2\sum_{j\geq0}\frac1{(2j+1)^2}
  =\frac{15}{128}\cdot\frac{\pi^2}4=\frac{15\pi^2}{512},
\]
and $e(0)=3J=45\pi^2/512$.
\end{proof}

\subsection{Truncated modes and the solution operators}\label{sec:modes}

We now proceed to solving the coupled trace equations \eqref{eq:left-system} and \eqref{eq:right-system}. We will proceed in generality for the moment.  Let $E(X,Y)$ denote a lower trace, $\ell$ or $q$, and let
$A(X,Y)=\Iop E$ be the corresponding upper trace, $a$ or $b$. For $m=2,3$
consider
\begin{equation}\label{eq:m-system}
 A_X=2\coth X\,A-2\csch X\,E,\qquad
 E_X-mE_Y=m\coth X\,E-m\csch X\,A.
\end{equation}
With $m=2$ this is the common homogeneous form of \eqref{eq:left-system}
and \eqref{eq:right-system}: the constant $-1$ in \eqref{eq:left-system}
is removed by subtracting a stationary solution, and the term $q_{y_4}$ in
\eqref{eq:right-system} is absorbed by a characteristic variable, as done
in Section~\ref{sec:region12}. On the bottom face $y_4=0$ of Region III, where
$q_X=3q_{y_4}$ by \eqref{eq:bottom-traces}, equation \eqref{eq:right-system}
becomes the case $m=3$. Since $A=\Iop E$ solves the first equation, the
second reads $E_Y=\frac1mE_X+\Kop E=\Lop_mE$ with $\Lop_m$ from
\eqref{eq:generators}. The problem is therefore: given $E(\cdot,0)=f$, find
$E(\cdot,Y)$. We write $\Sop_m(Y)f=E(\cdot,Y)$ for the propagator
of this Cauchy problem, and
$\Top_m(Y)f=\Iop\Sop_m(Y)f=A(\cdot,Y)$ for its upper-trace companion. We construct both by explicit formulas below; this produces the candidate formula, which Section~\ref{sec:verification} verifies directly. 

As a first step, it turns out we can easily solve the Cauchy problem with initial data given by $f =\Phi_t$---the mode defined in \eqref{eq:phi}.
\begin{proposition}\label{prop:opmodes}
Let $t>0$. The operators \eqref{eq:operators}--\eqref{eq:kop} act on a mode by
\begin{equation}\label{eq:mode-eigen}
 \Iop\Phi_t=\Phi_t^2,\qquad \Kop\Phi_t=-\coth t\,\Phi_t.
\end{equation}
Moreover, for $m=2,3$ the Cauchy problem \eqref{eq:m-system} with data
$E(\cdot,0)=\Phi_t$ is solved by products of truncated modes: with
$L=X+Y/m$ and $d=Y/m$, as functions of $X$,
\begin{equation}\label{eq:mode-transfer}
 \Sop_m(Y)\Phi_t=\Phi_t(L)\Phi_t(d)^m,\qquad
 \Top_m(Y)\Phi_t=\Phi_t(L)^2\Phi_t(d)^{m-1}.
\end{equation}
Both vanish for $t\leq L$, and at $Y=0$ they reduce to $\Phi_t$ and
$\Phi_t^2=\Iop\Phi_t$.
\end{proposition}
\begin{proof}
For $X\geq t$ both sides of both identities vanish, since $\Phi_t(s)=0$
for $s\geq t$. For $0<s<t$, differentiating the quotient
$\Phi_t(s)^2/(2\sinh^2s)=\sinh^2(t-s)/(2\sinh^2t\,\sinh^2s)$ and using
the addition formula $-\cosh(t-s)\sinh s-\sinh(t-s)\cosh s=-\sinh t$
gives
\[
 \frac{\diff}{\diff s}\,\frac{\Phi_t(s)^2}{2\sinh^2s}=-\frac{\Phi_t(s)}{\sinh^3s}.
\]
Since $\Phi_t(t)=0$, for $0<X<t$ this yields
\[
 \int_X^\infty\frac{\Phi_t(s)}{\sinh^3s}\,\diff s
 =\int_X^t\frac{\Phi_t(s)}{\sinh^3s}\,\diff s
 =\frac{\Phi_t(X)^2}{2\sinh^2X},
\]
and multiplying by $2\sinh^2X$ gives the first identity. Equivalently,
$\Phi_t^2$ is the bounded solution of the first trace equation
$g'=2\coth X\,g-2\csch X\,\Phi_t$. For the second identity,
\[
 \Kop\Phi_t=-\coth X\,\Phi_t+\csch X\,\Phi_t^2
 =\Phi_t\,\frac{\sinh(t-X)-\cosh X\sinh t}{\sinh X\,\sinh t}
 =-\coth t\,\Phi_t,
\]
since $\sinh(t-X)=\sinh t\cosh X-\cosh t\sinh X$ reduces the numerator
to $-\cosh t\sinh X$.
We now verify that $E=\Phi_t(L)\Phi_t(d)^m$ and
$A=\Phi_t(L)^2\Phi_t(d)^{m-1}$ satisfy \eqref{eq:m-system}. Put
$N_t(a)=\coth t\cosh a-\sinh a=\cosh(t-a)/\sinh t$, so that
$\Phi_t'(a)=-N_t(a)$ for $0<a<t$. Writing $t-d=(t-L)+X$ and
$t-L=(t-d)-X$ in the definition of $\Phi_t$, the addition formula for
$\sinh$ gives
\[
 \Phi_t(d)=\cosh X\,\Phi_t(L)+\sinh X\,N_t(L),\qquad
 \Phi_t(L)=\cosh X\,\Phi_t(d)-\sinh X\,N_t(d).
\]
Let $X>0$ and $L<t$. In the variables $(L,d)$ we have $X=L-d$,
$\partial_X=\partial_L$ and $\partial_Y=(\partial_L+\partial_d)/m$, so
$\partial_X-m\partial_Y=-\partial_d$. The two sides of the first equation
of \eqref{eq:m-system} are
\[
\begin{aligned}
 \partial_LA&=-2\Phi_t(L)N_t(L)\Phi_t(d)^{m-1},\\
 2\coth X\,A-2\csch X\,E
 &=\frac{2\Phi_t(L)\Phi_t(d)^{m-1}}{\sinh X}
   \bigl[\cosh X\,\Phi_t(L)-\Phi_t(d)\bigr],
\end{aligned}
\]
and the bracket equals $-\sinh X\,N_t(L)$ by the first addition
formula. The two sides of the second equation are
\[
\begin{aligned}
 -\partial_dE&=m\Phi_t(L)\Phi_t(d)^{m-1}N_t(d),\\
 m\coth X\,E-m\csch X\,A
 &=\frac{m\Phi_t(L)\Phi_t(d)^{m-1}}{\sinh X}
   \bigl[\cosh X\,\Phi_t(d)-\Phi_t(L)\bigr],
\end{aligned}
\]
and the bracket equals $\sinh X\,N_t(d)$ by the second addition formula.
For $L>t$ both $E$ and $A$ vanish identically, and since $\Phi_t(t)=0$
all four expressions above are continuous across $L=t$. The function $E$
itself is only Lipschitz across the front $L=t$: its derivative
$\partial_LE=\Phi_t'(L)\Phi_t(d)^m$ jumps from $-\Phi_t(d)^m/\sinh t$ to
$0$ there, while the combination $E_X-mE_Y=-\partial_dE$ is continuous.
Thus $E$ solves \eqref{eq:m-system} classically on either side of the
characteristic $L=t$ and as a Lipschitz solution across it; the
superpositions of modes with smooth data used below are classical
solutions. At $Y=0$ we have
$L=X$ and $d=0$, and $\Phi_t(0)=1$ gives $E=\Phi_t$ and
$A=\Phi_t^2=\Iop\Phi_t$. Finally $A$ is bounded, and the homogeneous
solutions of the first equation are the multiples of $\sinh^2X$, so the
first equation gives $A=\Iop E$ for every $Y$, as required of the upper
trace.
\end{proof}

Now, we recall that the truncated mode $\Phi_t$ of \eqref{eq:phi} is the
Green's function of $\partial_X^2-1$ with pole at $t$ that vanishes
for $X>t$, normalized by $\Phi_t(0)=1$.
By the Green's representation formula we have
\begin{equation}\label{eq:interpolation}
 f(X)=\int_X^\infty\sinh(t-X)\,(f''-f)(t)\dd
     =\int_0^\infty\sinh t\,(f''-f)(t)\,\Phi_t(X)\dd,
\end{equation}
for any smooth compactly supported $f$. 
Applying \eqref{eq:mode-transfer} under the integral in
\eqref{eq:interpolation} yields
\[
 \Sop_m(Y)f=\int_L^\infty\sinh t\,(f''-f)(t)\,\Phi_t(L)\Phi_t(d)^m\dd.
\]
The integral starts at $L$ because $\Phi_t(L)=0$ for $t<L$. Integrating
by parts twice moves the derivatives onto the kernel, and the resulting
expression involves $f$ alone, so it extends the solution operators to
bounded data. The next lemma records the resulting Green's function of the
Cauchy problem for a general product of modes; the same Green's function
produces the Region III formula in Section~\ref{sec:region3}.

\begin{lemma}\label{lem:universal-kernel}
Let $L>0$, let $a_1,\dots,a_N\in[0,L]$ with at least one $a_i$ equal to
$L$, and set $G(t)=\sinh t\prod_{i=1}^N\Phi_t(a_i)$.
\begin{enumerate}
\item[(i)] For $t>L$,
\begin{equation}\label{eq:kernel-density}
 G''-G=\frac2{\sinh^3t}\sum_{i<j}\sinh a_i\sinh a_j
       \prod_{h\notin\{i,j\}}\Phi_t(a_h),
\end{equation}
which is nonnegative and at most $N^2\sinh^2L/\sinh^3t$. Moreover
$G'(L)=\prod_{h\neq i}\Phi_L(a_h)$ if $a_i$ is the only entry equal to
$L$, and $G'(L)=0$ if two or more entries equal $L$.
\item[(ii)] For every smooth $f$ with compact support in $[0,\infty)$,
\begin{equation}\label{eq:universal-kernel}
 \mathcal P_{a_1,\dots,a_N}f
 =\int_L^\infty G(t)(f''-f)(t)\dd
 =G'(L)f(L)+\int_L^\infty (G''(t)-G(t))f(t)\dd,
\end{equation}
and the right side defines $\mathcal P_{a_1,\dots,a_N}f$ for every
bounded continuous $f$.
\item[(iii)] Let $L=X+Y/m$ and $d=Y/m$, and define
\begin{equation}\label{eq:transfer-definition}
 \Sop_m(Y)f=\mathcal P_{L,d,\dots,d}f,\qquad
 \Top_m(Y)f=\mathcal P_{L,L,d,\dots,d}f,
\end{equation}
with $m$ copies of $d$ in the first list and $m-1$ in the second. For
every bounded smooth $f$ on $[0,\infty)$, the functions $E=\Sop_m(Y)f$
and $A=\Top_m(Y)f$ solve \eqref{eq:m-system} on $X>0$, $Y\geq0$, and
\begin{equation}\label{eq:transfer-initial}
 \Sop_m(0)f=f,\qquad \Top_m(0)f=\Iop f,
 \qquad \Top_m(Y)=\Iop\Sop_m(Y).
\end{equation}
In particular $\partial_Y\Sop_m(Y)f=\Lop_m\Sop_m(Y)f$.
\end{enumerate}
\end{lemma}
\begin{proof}
For $0\leq a\leq t$, differentiating
$\Phi_t(a)=\sinh(t-a)/\sinh t$ in $t$ and using the addition formula gives
\[
 \partial_t\Phi_t(a)
 =\frac{\cosh(t-a)\sinh t-\sinh(t-a)\cosh t}{\sinh^2t}
 =\frac{\sinh a}{\sinh^2t}.
\]
Write $P=\prod_i\Phi_t(a_i)$, so that $G=\sinh t\,P$,
$G'=\cosh t\,P+\sinh t\,P'$ and $G''-G=2\cosh t\,P'+\sinh t\,P''$. By
the product rule and $\partial_t\csch^2t=-2\coth t\csch^2t$,
\[
\begin{aligned}
 P'&=\csch^2t\sum_i\sinh a_i\prod_{h\neq i}\Phi_t(a_h),\\
 P''&=-2\coth t\csch^2t\sum_i\sinh a_i\prod_{h\neq i}\Phi_t(a_h)
 +\csch^4t\sum_{i\neq j}\sinh a_i\sinh a_j\prod_{h\notin\{i,j\}}\Phi_t(a_h).
\end{aligned}
\]
In $2\cosh t\,P'+\sinh t\,P''$ the two single sums cancel, since
$2\cosh t\csch^2t=2\coth t\csch t=\sinh t\cdot2\coth t\csch^2t$, and
the sum over ordered pairs $i\neq j$ is twice the sum over $i<j$; this
gives \eqref{eq:kernel-density}. For $t\geq L$ every term is
nonnegative, because $\sinh a_i\geq0$ and
$\Phi_t(a_h)=\sinh(t-a_h)/\sinh t\geq0$ for $a_h\leq L\leq t$, and the
bound follows from $\sinh a_i\leq\sinh L$, $\Phi_t(a_h)\leq1$ and the
number $N(N-1)/2$ of pairs. For the atom, $\Phi_L(L)=0$ gives $P(L)=0$,
so $G'(L)=\sinh L\,P'(L)$. If $a_i$ is the only entry equal to $L$, every
term of $P'(L)$ other than the $i$th contains the factor $\Phi_L(a_i)=0$,
and the $i$th term is $\csch^2L\sinh L\prod_{h\neq i}\Phi_L(a_h)$, so
$G'(L)=\prod_{h\neq i}\Phi_L(a_h)$. If two entries equal $L$, every term
contains a vanishing factor and $G'(L)=0$.

Let $f$ be smooth with compact support. Two
integrations by parts on $[L,\infty)$ give
\[
 \int_L^\infty Gf''\dd
 =\Bigl[Gf'-G'f\Bigr]_{t=L}^{t=\infty}+\int_L^\infty G''f\dd
 =-G(L)f'(L)+G'(L)f(L)+\int_L^\infty G''f\dd,
\]
the terms at infinity vanishing because $f$ has compact support. Since
$G(L)=\sinh L\,P(L)=0$, subtracting $\int_L^\infty Gf\dd$ gives
\eqref{eq:universal-kernel}. The right side of
\eqref{eq:universal-kernel} makes sense for bounded continuous $f$
because $G''-G$ is integrable on $[L,\infty)$ by (i).

Let $f$ be smooth with compact support
in $[0,\infty)$ and put $c(t)=\sinh t\,(f''-f)(t)$, which is bounded with
compact support. Define
\[
 E(X,Y)=\int_0^\infty c(t)\,\Phi_t(L)\Phi_t(d)^m\dd,\qquad
 A(X,Y)=\int_0^\infty c(t)\,\Phi_t(L)^2\Phi_t(d)^{m-1}\dd.
\]
The integrands vanish for $t<L$, so by (ii), applied to the lists
$(L,d,\dots,d)$ and $(L,L,d,\dots,d)$, $E=\Sop_m(Y)f$ and
$A=\Top_m(Y)f$. For each $t$ the integrands solve \eqref{eq:m-system} by
Proposition~\ref{prop:opmodes}. They are Lipschitz in $(X,Y)$ with
constant bounded by a multiple of $\coth t$, since $|\Phi_t'|\leq\coth t$
on $[0,t]$, and $c(t)\coth t=\cosh t\,(f''-f)(t)$ is bounded, so
differentiation under the integral sign is permitted and $(E,A)$ solves
\eqref{eq:m-system}.

Let $f$ be bounded and smooth, and let $\chi_\nu$ be
smooth with $0\leq\chi_\nu\leq1$, $\chi_\nu=1$ on $[0,\nu]$ and $\chi_\nu=0$ on
$[\nu+1,\infty)$. The functions $f_\nu=\chi_\nu f$ are smooth with compact
support, so $E_\nu=\Sop_m(Y)f_\nu$ and $A_\nu=\Top_m(Y)f_\nu$ solve
\eqref{eq:m-system}. Let $K=G''-G$ denote the kernel of either list.
Where $L<\nu$ the atoms of $f$ and $f_\nu$ coincide, and $f=f_\nu$ on $[L,\nu]$,
so
\[
 \Sop_m(Y)f-E_\nu=\int_\nu^\infty K(t)\,(f-f_\nu)(t)\dd,
\]
and likewise for $\Top_m(Y)f-A_\nu$. The kernel $K$ and its first
derivatives in $L$ and $d$ are bounded by a multiple of $\sinh^{-3}t$,
locally uniformly in $(L,d)$, and the lower limit $\nu$ does not depend on
$(X,Y)$. Hence the differences and their first derivatives are bounded
by a multiple of $\|f\|_\infty\int_\nu^\infty\sinh^{-3}t\dd$, which tends to
zero uniformly on compact subsets of $\{X>0,\,Y\geq0\}$. Thus
$(E_\nu,A_\nu)\to(\Sop_m(Y)f,\Top_m(Y)f)$ together with first derivatives,
and \eqref{eq:m-system} passes to the limit.

We now prove \eqref{eq:transfer-initial}. At $Y=0$ we have
$L=X$ and $d=0$. For the list $(X,0,\dots,0)$ every pair $i<j$ contains
the factor $\sinh0=0$, so $G''-G\equiv0$, and $G'(X)=\Phi_X(0)^m=1$;
hence $\Sop_m(0)f=f$. For the list $(X,X,0,\dots,0)$ the atom vanishes,
the only surviving pair is the two copies of $X$, and
$G''-G=2\sinh^2X/\sinh^3t$; hence
$\Top_m(0)f=2\sinh^2X\int_X^\infty f(t)\sinh^{-3}t\dd=\Iop f$ by
\eqref{eq:operators}. Finally $A=\Top_m(Y)f$ tends to zero as
$X\to\infty$ with $Y$ fixed, since by (i)
$|A|\leq N^2\|f\|_\infty\,\sinh^2L\int_L^\infty\sinh^{-3}t\dd=\O(\eu^{-L})$.
The first equation of \eqref{eq:m-system} reads
$(A/\sinh^2X)_X=-2E/\sinh^3X$, and integrating it from $X$ to infinity
gives $A=\Iop E$, that is, $\Top_m(Y)=\Iop\Sop_m(Y)$. Substituting this
into the second equation gives $E_Y=\frac1mE_X+\Kop E=\Lop_mE$.
\end{proof}

In \eqref{eq:transfer-definition} the atom
$G'(L)f(L)=(\sinh X/\sinh L)^mf(L)$ carries the data along the
characteristic, with the integrating factor of the local part $-\coth X$
of $\Kop$, and the kernel \eqref{eq:kernel-density} collects the
contribution of the nonlocal part $\csch X\,\Iop$; together they form the
Green's function of the Cauchy problem. The name propagator is justified
by the semigroup law $\Sop_m(Y')\Sop_m(Y)=\Sop_m(Y+Y')$: since
$\Phi_t(X+d)=\Phi_t(d)\Phi_{t-d}(X)$, formula \eqref{eq:mode-transfer}
reads $\Sop_m(Y)\Phi_t=\Phi_t(d)^{m+1}\Phi_{t-d}$, so a mode is sent to a
multiple of the mode with pole $t-d$, and
$\Phi_t(d)\Phi_{t-d}(d')=\Phi_t(d+d')$ gives the law on modes and hence,
by superposition, on general data. Section~\ref{sec:region12} applies
$\Sop_2$ to the interface data in Regions I and II, and
Section~\ref{sec:region3} applies $\Sop_3$ followed by $\Sop_2$ in
Region III; the regional formulas then follow from
\eqref{eq:mode-transfer} mode by mode, and the positivity of the Green's
function of the Cauchy problem is used throughout
Section~\ref{sec:verification}.

\subsection{Regions I and II: the single-integral formula}\label{sec:region12}

We now derive the solution formula \eqref{eq:formula12} in Regions I and II. In this case, the five-expert formula is an additive perturbation of the four-expert formula established by Bayraktar, Ekren, and
Zhang~\cite[Theorem~3.1]{BEZ20}.
To illustrate the characteristic method used below for five experts, we first give a simple rederivation of the four-expert formula in the current notation; its $C^2$ regularity and the optimality of the COMB control are
taken from~\cite[Theorems~3.1 and~3.2]{BEZ20}.

For $x_1\geq x_2\geq x_3\geq x_4$, recall the coordinates
\eqref{eq:four-coordinates}, for which $z_1\geq z_2\geq|z_3|$. We write
$u_4=x_1+F_4(z_1,z_2,z_3)/k$ and seek a correction satisfying the
two fixed-direction equations
$\partial_{z_2}^2F_4=\partial_{z_3}^2F_4=F_4$.
These correspond to the controls $(1,0,1,0)$ and $(1,0,0,1)$.
Equality of adjacent coordinate derivatives on the three faces gives
\begin{equation}\label{eq:four-faces}
\begin{aligned}
 \partial_{z_1}F_4-\partial_{z_2}F_4&=0 &&(z_1=z_2),\\
 \partial_{z_2}F_4+\partial_{z_3}F_4&=-1 &&(z_3=-z_2),\\
 \partial_{z_2}F_4-\partial_{z_3}F_4&=0 &&(z_3=z_2).
\end{aligned}
\end{equation}
We first solve the equation in $z_3$, writing
$F_4=\alpha(z_1,z_2)\cosh z_3+\beta(z_1,z_2)\sinh z_3$, so that
$\partial_{z_2}F_4=\alpha_{z_2}\cosh z_3+\beta_{z_2}\sinh z_3$ and
$\partial_{z_3}F_4=\alpha\sinh z_3+\beta\cosh z_3$. On the faces
$z_3=-z_2$ and $z_3=z_2$ the last two conditions in
\eqref{eq:four-faces} read
\[
\begin{aligned}
 \alpha_{z_2}\cosh z_2-\beta_{z_2}\sinh z_2-\alpha\sinh z_2+\beta\cosh z_2&=-1,\\
 \alpha_{z_2}\cosh z_2+\beta_{z_2}\sinh z_2-\alpha\sinh z_2-\beta\cosh z_2&=0,
\end{aligned}
\]
and adding and subtracting them gives
\[
 \alpha_{z_2}\cosh z_2-\alpha\sinh z_2=-\tfrac12,
 \qquad
 \beta\cosh z_2-\beta_{z_2}\sinh z_2=-\tfrac12.
\]
Dividing by $\cosh^2z_2$ and by $\sinh^2z_2$, respectively, these are
$(\alpha/\cosh z_2)_{z_2}=-\tfrac12\sech^2z_2$ and
$(\beta/\sinh z_2)_{z_2}=\tfrac12\csch^2z_2$, so integration in $z_2$
gives
$\alpha=A(z_1)\cosh z_2-\sinh z_2/2$ and
$\beta=B(z_1)\sinh z_2-\cosh z_2/2$. Thus
\[
 F_4=A(z_1)\cosh z_2\cosh z_3
       +B(z_1)\sinh z_2\sinh z_3-\tfrac12\sinh(z_2+z_3).
\]
This also satisfies the equation in $z_2$. On the remaining face
$z_1=z_2>0$,
\[
\begin{aligned}
 \partial_{z_1}F_4-\partial_{z_2}F_4
 ={}&\bigl(A'\cosh z_2-A\sinh z_2+\tfrac12\cosh z_2\bigr)\cosh z_3\\
 &+\bigl(B'\sinh z_2-B\cosh z_2+\tfrac12\sinh z_2\bigr)\sinh z_3,
\end{aligned}
\]
and the coefficients of $\cosh z_3$ and $\sinh z_3$ must vanish
separately, since $-z_1\leq z_3\leq z_1$. With $z_2=z_1$ we obtain
\begin{equation}\label{eq:four-coefficient-odes}
 A'-\tanh z_1\,A=-\tfrac12,\qquad
 B'-\coth z_1\,B=-\tfrac12.
\end{equation}
The integrating factors are $\sech z_1$ and $\csch z_1$, respectively.
Boundedness of the correction fixes the constants of integration:
first take $z_2=z_3=0$ and let $z_1\to\infty$ to exclude a
homogeneous term proportional to $\cosh z_1$ in $A$; then take
$z_2=z_3>0$ fixed to exclude a term proportional to $\sinh z_1$
in $B$. Integrating from infinity therefore gives
\begin{equation}\label{eq:four-coefficients}
\begin{aligned}
 A(z_1)&=\tfrac12\cosh z_1\int_{z_1}^\infty\sech t\dd=\theta(z_1)\cosh z_1,\\
 B(z_1)&=\tfrac12\sinh z_1\int_{z_1}^\infty\csch t\dd=\tfrac12\lambda(z_1)\sinh z_1.
\end{aligned}
\end{equation}
Consequently $F_4$ is the function \eqref{eq:four-background}.
At $z_1=0$, the domain forces $z_2=z_3=0$, and the formula is
interpreted by continuity. In particular, the product containing
$\lambda(z_1)\sinh z_1$ tends to zero. To identify \eqref{eq:four-background}
with \cite[equation~(3.1)]{BEZ20}, reverse their ascending order,
$(x^{(1)},\dots,x^{(4)})=(x_4,x_3,x_2,x_1)$: their exponential argument
becomes $-z_1$, their three hyperbolic arguments become $-z_2,-z_3,z_1$,
$k(x_1-x_2)=z_2+z_3$, and $\atanh(\eu^{-z_1})=\lambda(z_1)/2$. Substitution
gives exactly \eqref{eq:four-background}. By \cite[Theorems~3.1
and~3.2]{BEZ20}, $u_4=x_1+F_4/k$ is the global $C^2$ solution of the
four-expert equation and $(1,0,1,0)$ attains its Hamiltonian maximum
throughout the closed ordered sector; Section~\ref{sec:verification} uses
these two facts.

To lift the four-expert solution to five experts, we need to use the density $p$ defined earlier in \eqref{eq:density}. The density is related to the trace by
\[
 p(X)=\tfrac12\tanh X-\eta(X).
\]
Indeed, by \eqref{eq:eta} the right side satisfies
\begin{equation}\label{eq:p-ode}
 p'+6\coth X\,p=\tfrac12\sech^2X,
\end{equation}
since $(\tfrac12\tanh X)'+3\coth X\tanh X=\tfrac12\sech^2X+3$ and
$\eta'+6\coth X\,\eta=3$. With the integrating factor $\sinh^6X$ this
reads $(\sinh^6X\,p)'=\tfrac12\sinh^6X\sech^2X$, and, exactly as for
$\eta$, the solution that is regular at zero is the positive integral in
\eqref{eq:density}. The elementary form
\eqref{eq:p-elementary} follows from $\eta=3\coth X-15I(X)/\sinh^6X$. In the derivation above, the four-expert version of the trace $e$ appeared as the coefficient $A(z_1)=\theta(z_1)\cosh z_1$ of \eqref{eq:four-coefficients} (it is the value of $F_4$ at $z_2=z_3=0$); since it is used throughout the rest of the paper, we give it the name $e_4$. Now, the corresponding interface endpoint and upper trace are
\begin{equation}\label{eq:four-traces}
 e_4(X)=\cosh X\,\theta(X),\qquad
 h_4(X)=\cosh^2X\,\theta(X)-\tfrac12\sinh X=\Iop e_4(X).
\end{equation}
Since $\eta=e-e''$ and $p=\tfrac12\tanh X-\eta$, we have
$e''-e=p-\tfrac12\tanh X$; and $\theta'=-\tfrac12\sech X$ gives
$e_4''-e_4=-\tfrac12\tanh X$, so the decaying difference $g=e-e_4$ obeys
$g''-g=p$. Its Green representation, which is the expansion of $g$ in
truncated modes with density $\sinh t\,p(t)$, is
\begin{equation}\label{eq:g-green}
 g(X)=\int_X^\infty\sinh(t-X)p(t)\dd
     =\int_0^\infty\sinh t\,p(t)\,\Phi_t(X)\dd.
\end{equation}
The integral converges, since \eqref{eq:density} gives
$p(t)=\O(\eu^{-2t})$. There is no free decaying homogeneous term: \eqref{eq:e-asymptotic} and the expansion of $e_4$ show
$g(X)=\O(\eu^{-2X})$, excluding both $\eu^X$ and $\eu^{-X}$ homogeneous terms.
Indeed, every other solution of $g''-g=p$ that decays at infinity differs
from \eqref{eq:g-green} by a multiple of $\eu^{-X}$, which the stronger decay $g=\O(\eu^{-2X})$
rules out, so no additional boundary assumption at zero is needed.

We now solve the equations at the interface where $Y=0$, or rather $y_1=y_3$. Here, the traces coincide pairwise by \eqref{eq:trace-defs}, $a=b=F(X,0,X,y_4)$ and
$\ell=q=F(0,X,0,y_4).$ We write $H(X,y_4)$ and $E(X,y_4)$ for these common values.
They satisfy
\begin{equation}\label{eq:interface-system}
 H=\Iop E,\qquad E_{y_4}=2\Kop E+1,\qquad E(X,0)=e(X).
\end{equation}
The first equation is $a=\Iop\ell$ of Section~\ref{sec:traces}, the
integrated form of the first equation in \eqref{eq:left-system}, at
$Y=0$. For the second equation, add the two equations \eqref{eq:trace-Y}. At $Y=0$
we have $\ell=q=E$, hence $\ell_X=q_X=E_X$, $\ell_{y_4}=q_{y_4}=E_{y_4}$, and
$\Kop\ell=\Kop q=\Kop E$, so by \eqref{eq:generators} the sum is
\[
 \ell_Y+q_Y=E_X+2\Kop E+\tfrac12-\tfrac12E_{y_4}.
\]
On the other hand, by \eqref{eq:trace-defs}, $\ell_Y=F_{y_3}(0,X,Y,y_4)$
and $q_Y=F_{y_1}(Y,X,0,y_4)$. Therefore at $Y=0$ both are derivatives of $F$ at
the point $(0,X,0,y_4)$, which lies on the faces $y_1=0$ and $y_3=0$. The
first and third conditions in \eqref{eq:faces} give
$F_{y_1}=\tfrac12(F_{y_2}-1)$ and $F_{y_3}=\tfrac12(F_{y_2}+F_{y_4})$
there, while $E(X,y_4)=F(0,X,0,y_4)$ gives $E_X=F_{y_2}$ and $E_{y_4}=F_{y_4}$.
Hence
\[
 \ell_Y+q_Y=E_X+\tfrac12E_{y_4}-\tfrac12.
\]
Equating the two expressions gives $E_{y_4}=2\Kop E+1$. Finally,
$E(X,0)=F(0,X,0,0)=\Psi(0,X)=e(X)$ by \eqref{eq:e-def}, and likewise
$H(X,0)=\Psi(X,0)=h(X)$, in agreement with $h=\Iop e$.

The pair $(h_4,e_4)$ of \eqref{eq:four-traces} is a stationary solution
of the first two equations in \eqref{eq:interface-system}: $h_4=\Iop e_4$
by \eqref{eq:four-traces}, and
\[
 \Kop e_4=-\coth X\,e_4+\csch X\,h_4
 =-\frac{\cosh^2X}{\sinh X}\,\theta(X)+\frac{\cosh^2X}{\sinh X}\,\theta(X)
  -\tfrac12=-\tfrac12.
\]
Therefore $E-e_4$ satisfies $(E-e_4)_{y_4}=2\Kop(E-e_4)$, with value $e-e_4=g$ at $y_4=0$.
Expanding $g$ in truncated modes by \eqref{eq:g-green} and using
$\Kop\Phi_t=-\coth t\,\Phi_t$ from \eqref{eq:mode-eigen}, each mode
evolves in $y_4$ by the factor $\eu^{-2y_4\coth t}$; applying
$\Iop\Phi_t=\Phi_t^2$, also from \eqref{eq:mode-eigen}, to the result
gives $H$. Thus
\begin{equation}\label{eq:interface-spectral}
\begin{aligned}
 E(X,y_4)&=e_4(X)+\int_X^\infty\sinh t\,p(t)\eu^{-2y_4\coth t}\Phi_t(X)\dd,\\
 H(X,y_4)&=h_4(X)+\int_X^\infty\sinh t\,p(t)\eu^{-2y_4\coth t}\Phi_t(X)^2\dd.
\end{aligned}
\end{equation}
Substituting these expressions directly into \eqref{eq:interface-system}
verifies the interface equations.

We now use the trace characteristics to extend the interface data into Regions I and II, by means of the propagator $\Sop_2$ of Lemma~\ref{lem:universal-kernel}.

In Region I, $y_1\leq y_3$, use $X=y_1+y_2$ and $Y=y_3-y_1$, as in
\eqref{eq:reconstruct-left}. The left system \eqref{eq:left-system}
contains no $y_4$-derivative, so $y_4$ is a parameter, and its second
equation is $\ell_Y=\Lop_2\ell+\tfrac12$ by \eqref{eq:trace-Y}, with the
data $\ell(\cdot,0,y_4)=E(\cdot,y_4)$ at the interface. The constant
$\tfrac12$ is removed by a stationary solution: the pair
\[
 E_3=\tfrac16(3+\eu^{-2X}),\qquad A_3=\tfrac23\eu^{-X}
\]
satisfies the first equation of \eqref{eq:left-system},
$A_3'=2\coth X\,A_3-2\csch X\,E_3$, and $A_3$ is bounded, so
$A_3=\Iop E_3$ as in Section~\ref{sec:traces}; a direct computation then
gives $\Lop_2E_3=\tfrac12E_3'-\coth X\,E_3+\csch X\,A_3=-\tfrac12$. (They are the traces $\ell$ and $a$ of the three-expert correction $\sqrt2(u-x_1)$ of \eqref{eq:3expert}, which explains the subscript.) Hence
$\ell-E_3$ solves the homogeneous Cauchy problem
$(\ell-E_3)_Y=\Lop_2(\ell-E_3)$ with the bounded smooth data
$E(\cdot,y_4)-E_3$, and Lemma~\ref{lem:universal-kernel}(iii) together with
$a=\Iop\ell$ from Section~\ref{sec:traces} gives
\begin{equation}\label{eq:left-solution}
\ell=E_3+\Sop_2(Y)\bigl(E(\cdot,y_4)-E_3\bigr),\qquad a=\Iop \ell.
\end{equation}

In Region II, $y_3\leq y_1\leq y_3+2y_4$, use $X=y_2+y_3$ and $Y=y_1-y_3$,
as in \eqref{eq:reconstruct-right}, and put $Z=y_4-Y/2$; the
condition $Z\geq0$ is the inequality $y_1\leq y_3+2y_4$ defining the
region. The right system \eqref{eq:right-system} contains $q_{y_4}$, and its
second equation is $q_Y+\tfrac12q_{y_4}=\Lop_2q$ by \eqref{eq:trace-Y}. The
operator $\partial_Y+\tfrac12\partial_{y_4}$ differentiates along the lines
$y_4-Y/2=Z$ of the $(Y,y_4)$-plane, so for fixed $Z$ the function
$\tilde q(X,Y)=q(X,Y,Z+Y/2)$ satisfies $\tilde q_Y=\Lop_2\tilde q$ with
data $\tilde q(\cdot,0)=q(\cdot,0,Z)=E(\cdot,Z)$: the backward
characteristic through $(Y,y_4)$ reaches the interface $Y=0$ at $y_4=Z$.
Lemma~\ref{lem:universal-kernel}(iii) and $b=\Iop q$ give
\begin{equation}\label{eq:right-solution2}
q=\Sop_2(Y)E(\cdot,Z),\qquad b=\Iop q.
\end{equation}
When $Z<0$ the characteristic meets the bottom face $y_4=0$ first; this is
Region III, treated in Section~\ref{sec:region3}.

We are now equipped to derive \eqref{eq:formula12}. In the coordinates \eqref{eq:four-coordinates}--\eqref{eq:coords12} of
Section~\ref{sec:main}, Region II is $z_3\geq0$, $z_4\geq0$ and Region I
is $z_3\leq0\leq z_4$. The characteristic coordinates of
Lemma~\ref{lem:universal-kernel} for $m=2$ are $L=X+Y/2$ and $d=Y/2$. In
Region II, $L=y_2+\tfrac12(y_1+y_3)=z_1$, $d=\tfrac12(y_1-y_3)=z_3$, and
$Z=y_4-Y/2=y_4-z_3=z_4$; in Region I, $L=y_2+\tfrac12(y_1+y_3)=z_1$,
$d=\tfrac12(y_3-y_1)=-z_3$, and $y_4=z_4$. With $d=|z_3|$, both regions
are therefore described by $0\leq d\leq z_2\leq z_1$ and $z_4\geq0$, and
the propagations \eqref{eq:left-solution} and \eqref{eq:right-solution2}
both start from the interface data $E(\cdot,z_4)$ of
\eqref{eq:interface-spectral}.

Consider Region II. By \eqref{eq:interface-spectral}, $E(\cdot,Z)$ is
$e_4$ plus a superposition of truncated modes with density
$\sinh t\,p(t)\eu^{-2Z\coth t}$, so \eqref{eq:mode-transfer}, applied mode
by mode in \eqref{eq:right-solution2}, gives
\begin{equation}\label{eq:region2-traces}
\begin{aligned}
 q&=\Sop_2(Y)e_4
   +\int_{z_1}^\infty\sinh t\,p(t)\eu^{-2z_4\coth t}\Phi_t(z_1)\Phi_t(z_3)^2\dd,\\
 b&=\Iop\Sop_2(Y)e_4
   +\int_{z_1}^\infty\sinh t\,p(t)\eu^{-2z_4\coth t}\Phi_t(z_1)^2\Phi_t(z_3)\dd,
\end{aligned}
\end{equation}
the integrals starting at $z_1$ because $\Phi_t(z_1)=0$ for $t<z_1$.
Reconstructing by \eqref{eq:reconstruct-right}, with $X=y_2+y_3$, the
integrands combine to
\[
 \frac{\sinh y_3\,\Phi_t(z_1)+\sinh y_2\,\Phi_t(z_3)}{\sinh(y_2+y_3)}
 \,\Phi_t(z_1)\Phi_t(z_3)=\Phi_t(z_1)\Phi_t(z_3)\Phi_t(z_2),
\]
because $z_1=z_2+y_2$ and $z_3=z_2-y_3$, so that with $r=t-z_2$ the
addition formula gives
\begin{equation}\label{eq:reconstruction-identity}
 \sinh y_3\sinh(r-y_2)+\sinh y_2\sinh(r+y_3)=\sinh r\,\sinh(y_2+y_3).
\end{equation}
This is the integral in \eqref{eq:formula12}. The four-expert parts of
\eqref{eq:region2-traces}, $\Sop_2(Y)e_4$ and $\Iop\Sop_2(Y)e_4$, are the
traces $q$ and $b$ of $F_4$ itself: $F_4$ is independent of $y_4$ and
satisfies the face conditions \eqref{eq:four-faces}, so its traces solve
the same Cauchy problem with the interface data $e_4$. Reconstructing them
returns $F_4(z_1,z_2,z_3)$, as one can also confirm by substituting
\eqref{eq:four-background} directly.

Region I is identical, with \eqref{eq:left-solution} and
\eqref{eq:reconstruct-left} in place of \eqref{eq:right-solution2} and
\eqref{eq:reconstruct-right}. The stationary pair $(E_3,A_3)$ belongs to
the four-expert part, $E_3+\Sop_2(Y)(e_4-E_3)$ being the trace $\ell$ of
$F_4$; the modes carry the factors $\Phi_t(z_1)\Phi_t(d)^2$ and
$\Phi_t(z_1)^2\Phi_t(d)$ with $d=-z_3$; and the addition formula holds
with $y_1$ in place of $y_3$ and $-z_3$ in place of $z_3$, since
$z_1=z_2+y_2$ and $-z_3=z_2-y_1$. Both regions therefore produce the
integrand of \eqref{eq:formula12}, in which $z_3$ enters only through
$|z_3|$, and the two formulas match on the interface $y_1=y_3$, where
$a=b$ and $\ell=q$ by \eqref{eq:trace-defs}. This proves
\eqref{eq:formula12}, which requires only one quadrature of elementary
functions.

\subsection{Region III: the finite formula}\label{sec:region3}
We now prove \eqref{eq:finite3}, in the coordinates \eqref{eq:coords3} of
Section~\ref{sec:main}, by propagating the slice trace $e$ along the
bottom face and then off it.

We begin with the propagation. In Region III, $y_1\geq y_3+2y_4$, we
keep the
Region II coordinates $X=y_2+y_3$ and $Y=y_1-y_3$ of
\eqref{eq:reconstruct-right} and put $W=Y-2y_4\geq0$; in the coordinates
\eqref{eq:coords3}, $W=3a_1$, $y_4=a_2-a_1$, and $X=a_4-a_2$. The right
system \eqref{eq:right-system} again gives $q_Y+\tfrac12q_{y_4}=\Lop_2q$ by
\eqref{eq:trace-Y}, but now $Z=y_4-Y/2<0$, so the backward characteristic
$y_4-Y/2=Z$ through $(Y,y_4)$ meets the bottom face $y_4=0$ before the
interface, at $Y=W$. On the bottom face the fourth face condition gives
$q_X=3q_{y_4}$ by \eqref{eq:bottom-traces}, so there, by
\eqref{eq:generators},
\[
 q_Y=\Lop_2q-\tfrac12q_{y_4}=\tfrac12q_X+\Kop q-\tfrac16q_X
    =\tfrac13q_X+\Kop q=\Lop_3q.
\]
The data at $Y=y_4=0$ is $q(X,0,0)=F(0,X,0,0)=\Psi(0,X)=e(X)$ by
\eqref{eq:trace-defs} and \eqref{eq:e-def}. Hence
Lemma~\ref{lem:universal-kernel}(iii) with $m=3$ gives
$q(\cdot,W,0)=\Sop_3(W)e$ on the bottom face, and with $m=2$, applied
along the characteristic as in Region II to
$\widehat q(X,s)=q(X,W+s,s/2)$ for $0\leq s\leq2y_4$, it gives
\begin{equation}\label{eq:right-solution3}
 q=\Sop_2(2y_4)\Sop_3(W)e,\qquad b=\Iop q,
\end{equation}
the second identity being $b=\Iop q$ of Section~\ref{sec:traces}.

We next compute the effect of the two propagators on a single mode.
Apply \eqref{eq:right-solution3} to a truncated
mode $\Phi_t$ with $t>a_4$. For $\Sop_3(W)$ the characteristic
coordinates of Lemma~\ref{lem:universal-kernel} are $L=X+W/3=X+a_1$ and
$d=W/3=a_1$, so \eqref{eq:mode-transfer} and the shift identity
$\Phi_t(r+a_1)=\Phi_t(a_1)\Phi_{t-a_1}(r)$ give
\[
 \Sop_3(W)\Phi_t=\Phi_t(X+a_1)\Phi_t(a_1)^3=\Phi_t(a_1)^4\,\Phi_{t-a_1}(X).
\]
For $\Sop_2(2y_4)$ they are $L=X+y_4$ and $d=y_4$. Applying
\eqref{eq:mode-transfer} to the mode $\Phi_{t-a_1}$ and using the shift
identity two more times, with $X+y_4+a_1=a_4$ and $y_4+a_1=a_2$, the traces of the mode are
\[
\begin{aligned}
 q&=\Phi_t(a_1)^4\,\Phi_{t-a_1}(X+y_4)\Phi_{t-a_1}(y_4)^2
   =\Phi_t(a_1)\Phi_t(a_4)\Phi_t(a_2)^2,\\
 b&=\Phi_t(a_1)^4\,\Phi_{t-a_1}(X+y_4)^2\Phi_{t-a_1}(y_4)
   =\Phi_t(a_1)\Phi_t(a_4)^2\Phi_t(a_2).
\end{aligned}
\]
Reconstructing by \eqref{eq:reconstruct-right}, the identity
\eqref{eq:reconstruction-identity} with $r=t-a_3$, since $a_4=a_3+y_2$
and $a_2=a_3-y_3$, gives
\[
 \frac{\sinh y_3\,\Phi_t(a_4)+\sinh y_2\,\Phi_t(a_2)}{\sinh(y_2+y_3)}
 =\Phi_t(a_3),
\]
so a mode produces the four-factor product $\prod_{i=1}^4\Phi_t(a_i)$.

We now pass from modes to the datum $e$ and obtain the kernel form of
$F$. For compactly supported smooth data $f$ in place
of $e$, the Green representation \eqref{eq:interpolation} and the
linearity of every step give
\[
 F=\int_{a_4}^\infty\sinh t\prod_{i=1}^4\Phi_t(a_i)\,(f''-f)(t)\dd
  =\mathcal P_{a_1,a_2,a_3,a_4}f
\]
by Lemma~\ref{lem:universal-kernel}(ii), the list $(a_1,a_2,a_3,a_4)$
having its entries in $[0,a_4]$ with $a_4=L$; the integral starts at
$a_4$ because $\Phi_t(a_4)=0$ for $t<a_4$. Both sides depend on $f$
through kernels of the type in Lemma~\ref{lem:universal-kernel}, and the
cutoff argument in its proof extends the identity from compactly
supported to bounded smooth data, so $F=\mathcal P_{a_1,a_2,a_3,a_4}e$.
Only the factor $\Phi_t(a_4)$ vanishes at $t=a_4$, so by
Lemma~\ref{lem:universal-kernel}(i) the atom is
$G'(a_4)=\prod_{i=1}^3\Phi_{a_4}(a_i)$ and the kernel
\eqref{eq:kernel-density} is a sum over the six pairs $i<j$:
\begin{equation}\label{eq:positive3}
\begin{aligned}
 F={}&e(a_4)\prod_{i=1}^3\Phi_{a_4}(a_i)\\
     &+2\int_{a_4}^\infty\frac{e(t)}{\sinh^3t}
       \sum_{i<j}\sinh a_i\sinh a_j
             \prod_{h\notin\{i,j\}}\Phi_t(a_h)\dd.
\end{aligned}
\end{equation}
This is the representation whose positivity Section~\ref{sec:verification}
uses.

Finally we derive the coefficients of the finite formula. Return to
compactly supported smooth data $f$, for which
$F[f]=\int_{a_4}^\infty\sinh t\prod_{i=1}^4\Phi_t(a_i)\,(f''-f)(t)\dd$,
and write the three factors $\Phi_t(a_i)=\cosh a_i-\coth t\,\sinh a_i$,
$i\leq3$, in the mode product. The product is multiaffine in the pairs
$(\cosh a_i,\sinh a_i)$, and collecting the terms with $j$ factors
$\sinh a_i$ gives
\[
 F[f]=\sum_{j=0}^3b_j[f]\,\sigma_j,\qquad
 b_j[f]=\int_{a_4}^\infty\sinh t\,(-\coth t)^j\,\Phi_t(a_4)\,(f''-f)(t)\dd,
\]
where $\sigma_0=\cosh a_1\cosh a_2\cosh a_3$,
$\sigma_3=\sinh a_1\sinh a_2\sinh a_3$, and $\sigma_1$, $\sigma_2$ are
the sums of three products displayed in \eqref{eq:finite3}. By the
eigenfunction identity \eqref{eq:mode-eigen},
$(-\coth t)^j\Phi_t=\Kop^j\Phi_t$, with $\Kop$ acting on the variable
$a_4$; since $f''-f$ has compact support, $\Kop^j$ may be taken outside
the integral, and the Green representation \eqref{eq:interpolation} gives
$b_j[f]=\Kop^jf$ evaluated at $a_4$. The integral defining $b_j[f]$
cannot be used with $e$ in place of $f$: since $e''-e=-\eta\to-\tfrac12$
and $\sinh t\,\Phi_t(a_4)=\sinh(t-a_4)$, its integrand would grow like
$\eu^{t}$. Integrating by parts twice instead, with
$\psi_j(t)=(-\coth t)^j\sinh(t-a_4)$, $\psi_j(a_4)=0$, and
$\psi_j'(a_4)=(-\coth a_4)^j$, gives
\[
 b_j[f]=(-\coth a_4)^jf(a_4)+\int_{a_4}^\infty(\psi_j''-\psi_j)(t)\,f(t)\dd,
\]
whose density $\psi_j''-\psi_j$ is $\O(\eu^{-3t})$ for fixed $a_4>0$; it vanishes
for $j=0$ and equals $2\sinh a_4\csch^3t$ for $j=1$. This form and the
kernel form $F[f]=\mathcal P_{a_1,a_2,a_3,a_4}f$ both pass to the bounded
datum $e$ under the cutoffs $f_\nu\to e$ used above, by dominated
convergence, and so does $\Kop^jf_\nu\to\Kop^je$, locally uniformly on
$(0,\infty)$, because $\Iop$ acts on bounded functions through an
absolutely convergent integral. Hence
$F=\sum_{j=0}^3b_j\sigma_j$ with $b_j=\Kop^je$ evaluated at $a_4$,
$j=0,1,2,3$, and $b_0=e$. It remains to compute $\Kop^je$ for $j\leq3$.

Two integration-by-parts identities are
\begin{equation}\label{eq:operator-identities}
 \partial_X\Iop=\Iop(\partial_X+\Kop),\qquad \Kop^2-[\partial_X,\Kop]=\mathrm{Id}.
\end{equation}
For the second identity, we set
\[J_j=\int_X^\infty f(t)\coth^jt/\sinh^3t\dd.\]
Both $\Kop^2f$ and $[\partial_X,\Kop]f$ contain
$2\cosh X\,J_0-6\sinh X\,J_1$; their multiplication terms are respectively
$\coth^2X\,f$ and $\csch^2X\,f$. This proves their difference is $f$.
The first identity follows by differentiating \eqref{eq:operators} and
integrating $\Iop\partial_Xf$ by parts.
Also, the antiderivatives
$\int\csch^3t\dd=-\tfrac12\csch t\coth t+\tfrac12\lambda(t)$ and
$\int\lambda(t)\csch^2t\dd=-\lambda(t)\coth t+\csch t$, which use
$\lambda'=-\csch t$, give $\Iop1=\cosh X-\lambda(X)\sinh^2X$ and
$\Iop(\lambda\sinh X)=2\lambda(X)\sinh X\cosh X-2\sinh X$, hence
\begin{equation}\label{eq:K-constants}
 \Kop1=-\lambda(X)\sinh X,\qquad \Kop^21=2-\lambda(X)\cosh X.
\end{equation}
The compatibility relation \eqref{eq:compat} is $\Kop e=\tfrac16(e'-3)$,
that is, $6b_1=e'-3$. The second identity in
\eqref{eq:operator-identities} says $\Kop f'=(\Kop f)'-\Kop^2f+f$. With
$f=e$ it gives $\Kop e'=b_1'-b_2+e$, so
$b_2=\Kop b_1=\tfrac16\Kop e'-\tfrac12\Kop1=\tfrac16(b_1'-b_2+e)-\tfrac12\Kop1$,
that is, $7b_2=b_1'+e-3\Kop1$. With $f=b_1$ it gives
$\Kop b_1'=b_2'-b_3+b_1$, so
$b_3=\Kop b_2=\tfrac17(\Kop b_1'+\Kop e-3\Kop^21)=\tfrac17(b_2'-b_3+2b_1-3\Kop^21)$,
that is, $8b_3=b_2'+2b_1-3\Kop^21$. Inserting \eqref{eq:K-constants} and
$(\lambda\sinh X)'=\lambda\cosh X-1$,
\[
 42b_2=e''+6e+18\lambda\sinh X,\qquad
 336b_3=e'''+20e'+144\lambda\cosh X-312,
\]
which are the coefficients \eqref{eq:coefficients} evaluated at $a_4$.
This completes the derivation of the finite formula \eqref{eq:finite3}.
In this representation, only $e(a_4)$ requires the quadrature
\eqref{eq:e-quadrature}; the derivatives of $e$ are eliminated by
\eqref{eq:e-first} and \eqref{eq:trace-derivatives}, the latter being
\eqref{eq:e-ode} and its derivative.

On the interface $z_4=0$ the two coordinate systems coincide, with
$a_1=0$, and $\Phi_t(0)=1$ removing the factor $\Phi_t(a_1)$, so
\eqref{eq:positive3} becomes the three-factor Green's representation
$\mathcal P_{a_2,a_3,a_4}e$ of the interface datum $e$. The Region II
formula \eqref{eq:formula12} at $z_4=0$ is the same representation
applied to $e=e_4+g$: by Section~\ref{sec:region12}, its four-expert
part is $F_4$ and its density part is the integral with $p=g''-g$, the
Laplace factor $\eu^{-2z_4\coth t}$ being $1$. The two formulas therefore
have the same value on the interface. Their integrands are different,
one involving $e$ and the other $p$ and $F_4$; the agreement of their
transverse derivatives up to second order is proved in
Section~\ref{sec:reg}.

\section{Verification, regularity, and the optimality set of COMB}\label{sec:verification}

The derivation of the five-expert solution formula given in Section \ref{sec:derivation} does not prove the formula is the viscosity solution of the expert PDE \eqref{eq:pde}. This section addresses this in multiple steps. We first verify in Section \ref{sec:reg} that the solution is $C^2$ across the sector boundaries, and that it has linear growth (bounded correction from the payoff). Then we verify in Sections \ref{sec:H3}, \ref{sec:H12}, and \ref{sec:H122} that $\v_*=(1,0,1,0,0)$ attains the Hamiltonian maximum throughout the whole ordered sector, which allows us to establish the formula as the unique viscosity solution of \eqref{eq:pde} in Section \ref{sec:comb}, and determine the exact optimality set of COMB (Theorem~\ref{thm:comb}). 

Many of the proofs in this section require tedious algebraic work, and
part of it is delegated to a computer. The division of labor is as
follows. The arguments that reduce Theorems~\ref{thm:main}
and~\ref{thm:comb} to finite computations are given in full in the text:
the face operators and the $C^2$ matching across the interfaces, the
symmetry and multiaffine reductions from the sector to a finite list of
corner cases, the sign principles of Lemmas~\ref{lem:laplace}
and~\ref{lem:cumulative}, the moment recursion, the integration of
differential sign identities from infinity, and the viscosity-solution
argument. What remains after these reductions is finite and of
two kinds. The first is a list of exact algebraic identities among explicit
hyperbolic, rational, and trace expressions, among them the face
conditions, the value and the first two transverse derivatives at the interfaces, and the control tables,
which comprise $64$ Hessian contractions in Region III and $64$ corner
gaps in Regions I and II. The second is the sign of $21$ explicit
scalar functions of one variable, $8$ in Region III and $13$ in Regions I
and II. Each has the form $a(L)e(L)+b(L)$ with $a$ rational in
$\varrho=\eu^L$ and $b$ rational in $\varrho$ and linear in $L$ and
$\lambda(L)$; the differential sign identities of the form
\eqref{eq:first-scalar} and \eqref{eq:second-scalar3} eliminate the trace
$e$, and their residuals are then of the rational-log form
\eqref{eq:rational-log-form}.

Both kinds are verified by the computational supplement described in
Section~\ref{sec:certificates}, and the text states at each such point
which check of the supplement verifies it. The identities are checked
symbolically: after the substitution $\varrho=\eu^L$, and its analogues for
the other coordinates, every hyperbolic identity becomes an identity of
rational functions, which is decided by exact polynomial arithmetic. Each
identity is also derived in reduced form in the text, so any one of them
can be verified by hand. The scalar signs are established by exact
certificates: the logarithms in \eqref{eq:rational-log-form} are enclosed
between rational functions on the seven intervals \eqref{eq:boxes}, and
the resulting $147$ polynomial lower bounds are shown to be nonnegative
through their Bernstein coefficients \eqref{eq:bernstein}, which are exact
rational numbers. No floating-point arithmetic, numerical quadrature, or
sign sampling is used anywhere in the proof, so there are no rounding
errors to control. The trusted software consists of the exact integer and
rational arithmetic of CPython and the exact algebra and polynomial root
counting of SymPy, and the certificates are finite lists of integers that
the supplement rechecks with a separate program using only fraction
arithmetic. In this sense the supplement's part of the proof is an
exact symbolic computation with the same logical status as a long hand
calculation; the supplement is not a formalization in a proof assistant, and the separate Lean~4 formalization of this verification is described at the end of Section~\ref{sec:certificates}. The supplement is a
fixed, versioned release with pinned dependencies that runs offline in a
few minutes; its review guide maps every check to an equation label of
this paper and lists the analytic steps that remain for the reader.

\subsection{Regularity and boundedness} \label{sec:reg}

The aim of this section is to prove regularity and boundedness of the solution formula for $u$. In particular, we prove the following result.
\begin{proposition}\label{prop:C2}
Let $u$ be the function of Theorem~\ref{thm:main}: on the ordered sector
$u=x_1+F(y)/k$ as in \eqref{eq:scaling}, with $F$ given by
\eqref{eq:formula12} in Regions I and II and by \eqref{eq:finite3} in
Region III, and $u$ is extended to $\R^5$ by sorting the coordinates.
Then $u\in C^2(\R^5)$, satisfies
$\nabla u \cdot \one = 1$, and has bounded correction to $\phi$. Furthermore, in 
the ordered sector, $D_{\v_*}^2u=2(u-x_1)$.
\end{proposition}

The proof of Proposition \ref{prop:C2} contains many steps that occupy the rest of this section. Throughout we write $d=|z_3|$. In Region I, where
$d=-z_3=\tfrac12(y_3-y_1)$ and $z_4=y_4$, and in Region II, where
$d=z_3$ and $z_4=y_4-z_3$, the chain rule for the coordinates
\eqref{eq:four-coordinates}--\eqref{eq:coords12} gives
\begin{equation}\label{eq:chain-12}
\begin{aligned}
 \text{I:}\quad&\partial_{y_1}=\tfrac12(\partial_{z_1}+\partial_{z_2}-\partial_d),&
 \partial_{y_3}&=\tfrac12(\partial_{z_1}+\partial_{z_2}+\partial_d),\\
 \text{II:}\quad&\partial_{y_1}=\tfrac12(\partial_{z_1}+\partial_{z_2}+\partial_d-\partial_{z_4}),&
 \partial_{y_3}&=\tfrac12(\partial_{z_1}+\partial_{z_2}-\partial_d+\partial_{z_4}),
\end{aligned}
\end{equation}
with $\partial_{y_2}=\partial_{z_1}$ and $\partial_{y_4}=\partial_{z_4}$ in
both regions, and in Region III, by \eqref{eq:coords3},
\begin{equation}\label{eq:chain-III}
\begin{aligned}
 \partial_{y_1}&=\tfrac13(\partial_{a_1}+\partial_{a_2}+\partial_{a_3}+\partial_{a_4}),&
 \partial_{y_2}&=\partial_{a_4},\\
 \partial_{y_3}&=\tfrac13(-\partial_{a_1}-\partial_{a_2}+2\partial_{a_3}+2\partial_{a_4}),&
 \partial_{y_4}&=\tfrac13(-2\partial_{a_1}+\partial_{a_2}+\partial_{a_3}+\partial_{a_4}).
\end{aligned}
\end{equation}
In particular the direction $b_{\v_*}=(1,-1,1,0)$ of \eqref{eq:fixed} is
$\partial_{y_1}-\partial_{y_2}+\partial_{y_3}=\partial_{z_2}$ in Regions I
and II and $\partial_{a_3}$ in Region III, and the displayed formulas
satisfy the corresponding second-derivative equation immediately: every
term of \eqref{eq:finite3} has degree one in $(\cosh a_3,\sinh a_3)$, and
in \eqref{eq:formula12} both $F_4$ and $\Phi_t(z_2)$ satisfy $f''=f$ in
$z_2$. This proves \eqref{eq:fixed} in each regional interior.
One-variable
functions such as
$e$, $p$, $\theta$, $\lambda$, the moments below, and the scalar
certificates are written in a generic variable $L>0$, which stands for
$z_1$ in Regions I and II and for $a_4$ in Region III. We abbreviate
$S=\sinh L$ and $C=\cosh L$.

The density $p$ enters every estimate below through the following bounds.
\begin{lemma}\label{lem:p-bounds}
For $L>0$,
\begin{equation}\label{eq:p-bounds}
 \frac1{14}\tanh L\,\sech^2L\leq p(L)
 \leq\frac18\tanh L\,\sech^2L.
\end{equation}
Moreover $p(L)\sim L/14$ at zero and $p(L)\sim \eu^{-2L}/2$ at infinity.
\end{lemma}
\begin{proof}
We apply $\partial_L+6\coth L$ to each of the proposed barriers and subtract
the right side $\tfrac12\sech^2L$ of \eqref{eq:p-ode}. For the lower and
upper barriers, respectively, this gives
\[
 -\tfrac3{14}\tanh^2L\,\sech^2L\leq0
 \qquad\text{and}\qquad
 \tfrac38\sech^4L\geq0.
\]
Multiplying by the positive integrating factor $\sinh^6L$ and integrating
from zero proves the two bounds. The asymptotic statements follow from
\eqref{eq:density}. Near zero the integrand is $t^6+\O(t^8)$, so
\[
 p(L)=\frac{L^7/7+\O(L^9)}{2L^6+\O(L^8)}\sim\frac L{14}.
\]
At infinity
$\sinh^6t\,\sech^2t\sim \eu^{4t}/16$, so the integral is asymptotic to
$\eu^{4L}/64$ while $2\sinh^6L\sim \eu^{6L}/32$, and $p(L)\sim \eu^{-2L}/2$.
In particular the lower bound in \eqref{eq:p-bounds} is sharp at zero and
the upper bound is sharp at infinity.
\end{proof}
We next verify the four boundary conditions \eqref{eq:faces} on the faces
of the ordered sector directly from the formula.
Section~\ref{sec:derivation} used these conditions only through the trace
systems \eqref{eq:left-system}--\eqref{eq:bottom-traces}, which were solved as
Cauchy problems from data on an edge, imposing the remaining conditions only
on that edge; it also assumed growth conditions at infinity and smooth
traces. Since each equation in \eqref{eq:faces} states that the normal
derivative of $u$ vanishes on a permutation hyperplane, \eqref{eq:faces} is part of
the $C^2$ claim of Theorem~\ref{thm:main} and must be established for the
final formulas. We use \eqref{eq:chain-12} to express each face operator in
the regional coordinates.
The four-expert term of \eqref{eq:formula12} is $F_4(z_1,z_2,-d)$ in
Region I and $F_4(z_1,z_2,d)$ in Region II, with $F_4$ from
\eqref{eq:four-background}; we write $F_4^\mathrm I$ and $F_4^\mathrm{II}$
for these. The correction in \eqref{eq:formula12} is
symmetric in $d,z_2$. On $y_1=0$ in Region I, $d=z_2$ and
$2\partial_{y_1}-\partial_{y_2}=\partial_{z_2}-\partial_d$, which
annihilates every function symmetric in $(d,z_2)$ on the diagonal $d=z_2$.
The only asymmetric term of $F$ is $-\tfrac12\sinh(z_2-d)$ in
$F_4^\mathrm I$, and
$(\partial_{z_2}-\partial_d)\bigl(-\tfrac12\sinh(z_2-d)\bigr)=-\cosh(z_2-d)=-1$
at $d=z_2$, which is the required value.
On $y_2=0$ in Regions I and II, $z_1=z_2$ and the face operator
$-\partial_{y_1}+2\partial_{y_2}-\partial_{y_3}$ is
$\partial_{z_1}-\partial_{z_2}$, the $\partial_d$ and $\partial_{z_4}$
terms cancelling. The background satisfies this identity by its face
condition \eqref{eq:four-faces}, and the correction does too by symmetry of
$\Phi_t(z_1)\Phi_t(z_2)$: its moving-endpoint term is zero since
$\Phi_{z_1}(z_1)=0$. On $y_3=0$ in Region II, $d=z_2$ and the operator
$-\partial_{y_2}+2\partial_{y_3}-\partial_{y_4}$ is
$\partial_{z_2}-\partial_d$, so both terms vanish by symmetry.

It remains to check the boundary condition on $y_4=0$ in Region I.
Here $z_4=0$, and by \eqref{eq:chain-12} the condition
$-F_{y_3}+2F_{y_4}=0$ reads
\begin{equation}\label{eq:bottom-face-direct}
 4F_{z_4}=F_{z_1}+F_d+F_{z_2},\qquad 0\leq d\leq z_2\leq z_1.
\end{equation}
Put $c=\coth t$, $P=\Phi_t(z_1)\Phi_t(d)\Phi_t(z_2)$, and
$N_t(a)=c\cosh a-\sinh a=-\partial_a\Phi_t(a)$. Since $F_4^\mathrm I$
does not depend on $z_4$, while in the correction $\partial_{z_4}$
produces the factor $-2c$, the tangential derivatives act on $P$ through
$\partial_a\Phi_t(a)=-N_t(a)$, and the moving lower limit contributes
nothing because $\Phi_{z_1}(z_1)=0$, condition
\eqref{eq:bottom-face-direct} at $z_4=0$ is equivalent to
\begin{align*}
 (\partial_{z_1}+\partial_d+\partial_{z_2})F_4^\mathrm I
 =\int_{z_1}^\infty\sinh t\,p(t)\bigl[&-8cP
       +N_t(z_1)\Phi_t(d)\Phi_t(z_2)\\
       &+\Phi_t(z_1)N_t(d)\Phi_t(z_2)
       +\Phi_t(z_1)\Phi_t(d)N_t(z_2)\bigr]\dd.
\end{align*}
The bracket is a polynomial in $c$ of degree at most four. We replace $c^j$
by the moments $M_j(z_1)$ from \eqref{eq:moments01}--\eqref{eq:momentrec},
for $0\leq j\leq4$. Substituting the moment formulas gives the left side exactly; no
additional relation between $e$ and $p$ is needed. We verify this finite
symbolic identity in \path{checks/faces.py} in the proof supplement.

In Region III the face operators on $y_2=0$, $y_3=0$, and $y_4=0$ are,
respectively, $\partial_{a_4}-\partial_{a_3}$ at $a_3=a_4$,
$\partial_{a_3}-\partial_{a_2}$ at $a_2=a_3$, and
$\partial_{a_2}-\partial_{a_1}$ at $a_1=a_2$.
The last two identities follow from symmetry of \eqref{eq:finite3}.
The first follows by substituting \eqref{eq:coefficients} and the trace
ODE into the derivative of that finite formula; it is checked in the same
audit. These direct checks cover every relatively open physical face.
Their intersections follow from the continuity of the value, gradient, and
Hessian up to the boundary of the sector, established below.

We now show that the correction to the payoff is bounded. This follows
directly from the integral representations. The four-expert correction is
bounded, and the correction integral in \eqref{eq:formula12} is bounded above
by $\int_0^\infty\sinh t\,p(t)\dd<\infty$. In \eqref{eq:positive3},
each $\Phi_{a_4}(a_i)\in[0,1]$, every $\Phi_t(a_i)\leq1$, and $\sinh a_i\leq\sinh a_4$.
Also
\[
 2\sinh^2a_4\int_{a_4}^\infty\frac{\dd}{\sinh^3t}
 \leq 2\sinh^2a_4\int_{a_4}^\infty\frac{\cosh t}{\sinh^3t}\dd=1.
\]
There are six pairs in the kernel, so $0\leq F\leq7\|e\|_\infty$ in
Region III. In particular $u-\phi$ is bounded throughout the sector.
This is the one place where Section~\ref{sec:verification} uses
Section~\ref{sec:derivation} beyond the definition of $e$: the upper
bound in Region III rests on the equality of \eqref{eq:finite3} with
\eqref{eq:positive3}, established in Section~\ref{sec:region3} by the
passage from compactly supported data to $e$, and it is not visible from
\eqref{eq:finite3} alone, whose hyperbolic products grow like
$\eu^{a_1+a_2+a_3}$. The lower bound $F\geq0$ in Region III also follows
from Section~\ref{sec:H3}: the gap of the control $00000$ is $F$ itself,
and its vertex values $e$, $C_1+3T_1+4D_1$, $3(A_2+B_2+T_2)$, and
$3T_3+6P_3$ in the tables there are nonnegative.

The formula in each region is smooth for $z_1>0$, respectively $a_4>0$, up to its relative boundary.
We show next that the values and derivatives agree across the two interfaces.

At the first interface keep $(z_1,z_2,y_4)$ fixed and vary the signed
coordinate $z_3=(y_1-y_3)/2$. The
four-expert term \eqref{eq:four-background} is already smooth in $z_3$.
For a fixed $c=\coth t$, the only changing factor in the correction is
$\eu^{-2z_4c}\Phi_t(|z_3|)$, where $z_4=y_4$ in Region I and $z_4=y_4-z_3$ in
Region II, and $\Phi_t(\mp z_3)=\cosh z_3\pm c\sinh z_3$ by the addition
formula; that is,
\begin{equation}\label{eq:first-join-factors}
 \eu^{-2y_4c}
 \begin{cases}
  \cosh z_3+c\sinh z_3,&z_3\leq0,\\
  \eu^{2cz_3}(\cosh z_3-c\sinh z_3),&z_3\geq0.
 \end{cases}
\end{equation}
Both branches have value $1$, first derivative $c$, and second derivative
$1$ at $z_3=0$: for the second branch the derivative is
$\eu^{2cz_3}\bigl[c\cosh z_3+(1-2c^2)\sinh z_3\bigr]$, whose value and
derivative at zero are $c$ and $2c^2+(1-2c^2)=1$. Since the integration
endpoint $z_1$ is fixed, this proves
equality of the value and the first two transverse derivatives as functions
of the tangential variables. These three identities hold at every point of
the interface, so they may be differentiated with respect to $z_1$, $z_2$,
and $y_4$; this yields the tangential first derivatives and the
tangential--tangential and mixed second derivatives, and hence the value,
gradient, and Hessian of the two regional formulas agree at $y_1=y_3$.

For the second interface use the Region III coordinates on both sides and
write $\tau=a_1$, keeping $(a_2,a_3,a_4)$ fixed. By the relations
$a_1=-\tfrac23z_4$ and $a_i=z_{5-i}+\tfrac13z_4$ recorded after
\eqref{eq:inverse3}, the Region II coordinates become
\begin{equation}\label{eq:second-join-coordinates}
 (z_1,z_2,z_3,z_4)=(a_4+\tau/2,a_3+\tau/2,a_2+\tau/2,-3\tau/2),
\end{equation}
so that $\partial_\tau z_i=\tfrac12$ for $i\leq3$ and
$\partial_\tau z_4=-\tfrac32$.
Define
\begin{align*}
 \mathcal B_0&=b_0\cosh a_2\cosh a_3+b_1\sinh(a_2+a_3)+b_2\sinh a_2\sinh a_3,\\
 \mathcal B_1&=b_1\cosh a_2\cosh a_3+b_2\sinh(a_2+a_3)+b_3\sinh a_2\sinh a_3.
\end{align*}
Formula \eqref{eq:finite3} is $F=\cosh\tau\,\mathcal B_0+\sinh\tau\,\mathcal B_1$ on
the Region III side. On the Region II side, direct differentiation gives
\begin{equation}\label{eq:second-join-jets}
 \left.(F,F_\tau,F_{\tau\tau})\right|_{\tau=0^-}=(\mathcal B_0,\mathcal B_1,\mathcal B_0).
\end{equation}
To verify this identity, we must account for the moving endpoint of the
integral. We set $P=\Phi_t(a_4)\Phi_t(a_2)\Phi_t(a_3)$ and
\[
 \mathcal D_0=\tfrac12(\partial_{a_4}+\partial_{a_2}+\partial_{a_3}),
 \qquad \mathcal D_c=\mathcal D_0+3c.
\]
By \eqref{eq:second-join-coordinates}, $\partial_\tau$ acts as $\mathcal D_0$ on
$F_4(z_1,z_2,z_3)$. The integrand of \eqref{eq:formula12} is
$\sinh t\,p(t)\eu^{-2z_4c}$ times $\Phi_t(z_1)\Phi_t(z_3)\Phi_t(z_2)$, and
since $\partial_\tau \eu^{-2z_4c}=3c\eu^{-2z_4c}$, the operator $\partial_\tau$
acts on it as $\mathcal D_c$. The lower limit $z_1=a_4+\tau/2$
moves as well. Its first derivative contributes nothing, because the
integrand vanishes at $t=z_1$; its second derivative contributes
$-\tfrac12\,\partial_\tau(\text{integrand})|_{t=z_1}$, where only the
factor $\Phi_t(z_1)$ has a nonzero derivative, namely
$\partial_\tau\Phi_t(z_1)=\tfrac12\Phi_t'(z_1)$, which equals
$-1/(2\sinh t)$ at $t=z_1$. For $j=0,1,2$ the left side of \eqref{eq:second-join-jets} is
therefore
\begin{equation}\label{eq:jet-recipe}
 \mathcal D_0^jF_4+\int_{a_4}^\infty\sinh t\,p(t)\mathcal D_c^jP\dd,
\end{equation}
plus, for $j=2$, the endpoint term
$\tfrac14p(a_4)\Phi_{a_4}(a_2)\Phi_{a_4}(a_3)$, which is
$(-\tfrac12)\cdot\sinh a_4\,p(a_4)\cdot(-\tfrac1{2\sinh a_4})\Phi_{a_4}(a_2)\Phi_{a_4}(a_3)$
at $\tau=0$, where $z_4=0$. Expand $\mathcal D_c^jP$ as a polynomial in $c$ and use the moments
\eqref{eq:moments01}--\eqref{eq:momentrec} below. Substitution of
\eqref{eq:p-elementary} and \eqref{eq:trace-derivatives} gives respectively
$\mathcal B_0,\mathcal B_1,\mathcal B_0$, coefficient by coefficient in the four products
of $\cosh a_2$ or $\sinh a_2$ with $\cosh a_3$ or $\sinh a_3$.
The symbolic audit also checks this identity directly.
The value and first two transverse derivatives agree as functions of all
three tangential variables, and tangential differentiation, as at the
first interface, gives equality of the value, gradient, and Hessian of
the Region II and Region III formulas on the second interface.

It remains to show that the derivatives extend continuously to $z_1=0$ in Regions I and II and to $a_4=0$ in Region III. In Regions I and II, Lemma~\ref{lem:p-bounds} gives
\[
 \sinh t\,p(t)=\O(t^2)\quad(t\downarrow0),\qquad
 \sinh t\,p(t)=\O(\eu^{-t})\quad(t\to\infty).
\]
For $0\leq a\leq z_1\leq t$, we have $0\leq\Phi_t(a)\leq1$ and
$|\partial_a\Phi_t(a)|\leq\coth t$. Every derivative of order at most two
of the integrand in \eqref{eq:formula12} is therefore bounded by a constant multiple of
\[
 \sinh t\,p(t)(1+\coth^2t),
\]
uniformly on the physical domain, since $z_4\geq0$. This is integrable both
at zero and infinity. The first endpoint term vanishes because $\Phi_{z_1}(z_1)=0$;
the second endpoint terms are $\O(p(z_1))=\O(z_1)$. Dominated convergence proves
continuous limits for the function, gradient, and Hessian at $z_1=0$, including
the top-four collision with arbitrary $y_4\geq0$.
The four-expert term has the same property: its only nonanalytic factor at
zero is $\lambda(z_1)\sinh z_1\sinh z_2\sinh z_3$, whose second derivatives tend to zero.

In Region III the trace $e$ is analytic at zero. The nonanalytic part of
\eqref{eq:finite3} is a linear combination of
\[
 \lambda(a_4)\sinh a_4\sum_{i<j}\sinh a_i\sinh a_j\cosh a_h,
 \qquad \lambda(a_4)\cosh a_4\sinh a_1\sinh a_2\sinh a_3,
\]
where $h$ is the remaining index. On $0\leq a_i\leq a_4$ these terms are
$\O(a_4^3|\log a_4|)$ and all their second partial derivatives are
$\O(a_4|\log a_4|)$. The remaining terms are analytic. The value, gradient,
and Hessian therefore have unique limits at the origin in each region. The
interface identities prove that those limits agree: approach the origin
along either common interface, where the value, gradient, and Hessian of
the two regional formulas have already been shown to agree.

To pass from these one-sided identities to $C^2$ regularity we use the
following gluing lemma, whose proof is standard and given in Appendix~\ref{sec:appendix-gluing}.
\begin{lemma}\label{lem:gluing}
Let $\Omega\subset\R^n$ be open and convex, let $K_1,\dots,K_N$ be closed
convex sets with nonempty interiors whose union contains $\Omega$, and let
$u\colon\Omega\to\R$ be continuous. Suppose that for each $j$ the
restriction of $u$ to the intersection of $\Omega$ with the interior of
$K_j$ is $C^2$, that its gradient and Hessian extend continuously to
$\Omega\cap K_j$, and that at every point of $\Omega$ the extended
gradients and Hessians of all the $K_j$ containing that point agree. Then
$u\in C^2(\Omega)$.
\end{lemma}

We first apply the lemma inside the sector. Let $\mathcal C$ be the closed ordered
sector in $\R^5$, write $\bar u$ for the function defined on $\mathcal C$ by the
formula of Theorem~\ref{thm:main}, and take $\Omega$ to be the interior of
$\mathcal C$, where all $y_i>0$, and $K_1,K_2,K_3$ the closures of Regions I, II,
and III, which are closed convex cones. Inside each region the formula is
smooth, as noted before the interface computations, and its gradient and
Hessian extend continuously to the closed region: in Regions I and II by
the dominated convergence bound just established, which is uniform on the
physical domain, together with the $C^2$ regularity of $F_4$, and in
Region III because \eqref{eq:finite3} is analytic for $a_4>0$ and has the
limits at the origin found above. A point of the open sector lying in two
regions lies on an interface, and a point lying in all three lies on both
interfaces, at $z_3=z_4=0$; the interface identities give the agreement
of the extended gradients and Hessians at such points, and the agreement
of the values makes $\bar u$ continuous. Lemma~\ref{lem:gluing} therefore
shows that $\bar u$ is $C^2$ in the interior of $\mathcal C$. Moreover
$\nabla\bar u$ and $\nabla^2\bar u$ extend continuously to $\mathcal C$: at
boundary points with $z_1>0$ by the extensions from the closed regions,
and at $z_1=0$ and at the origin by the limits established above, where
the extensions from different regions agree by continuity, since they
agree on the interfaces.

We now extend $\bar u$ to $\R^5$ by permutation invariance. For
$x\in\R^5$ let $x^\downarrow$ be the vector obtained by sorting the
coordinates of $x$ in decreasing order, and set
$u(x)=\bar u(x^\downarrow)$. This is well defined and continuous, since
the $i$th coordinate of $x^\downarrow$, the $i$th largest coordinate of
$x$, is a continuous function of $x$. Let each permutation $\sigma$ of
$\{1,\dots,5\}$ act on $\R^5$ by permuting coordinates, and write $\sigma$
also for the corresponding orthogonal matrix. The sectors $\sigma(\mathcal C)$ are
closed convex cones with nonempty interiors covering $\R^5$, and on
$\sigma(\mathcal C)$ we have $u=\bar u\circ\sigma^{-1}$, whose gradient and Hessian
are $\sigma\,\nabla\bar u\circ\sigma^{-1}$ and
$\sigma\,\nabla^2\bar u\circ\sigma^{-1}\,\sigma^\top$ in the interior and
extend continuously to $\sigma(\mathcal C)$. To apply Lemma~\ref{lem:gluing} with
$\Omega=\R^5$ it remains to check the agreement hypothesis, and by
permutation invariance it suffices to do so at points $x\in\mathcal C$. The
sectors containing $x$ are exactly the $\sigma(\mathcal C)$ with $\sigma x=x$:
indeed $x\in\sigma(\mathcal C)$ means that $\sigma^{-1}x$ is sorted in decreasing
order, that is, $\sigma^{-1}x=x^\downarrow=x$. On such a sector the
extended gradient and Hessian at $x$ are $\sigma\nabla\bar u(x)$ and
$\sigma\nabla^2\bar u(x)\sigma^\top$, so the agreement hypothesis at $x$
states that $\nabla\bar u(x)$ and $\nabla^2\bar u(x)$ are invariant under
every $\sigma$ with $\sigma x=x$. These permutations preserve each block
of tied coordinates of $x$, and they are generated by the transpositions
$\sigma_i$ of $i$ and $i+1$ over the indices $i$ with $x_i=x_{i+1}$, so
it suffices to check invariance under these. Fix such an $i$ and let
$\nu\in\R^5$ be the unit vector with $\nu_i=1/\sqrt2$,
$\nu_{i+1}=-1/\sqrt2$, and all other components zero, so that
$\sigma_i=\mathrm{Id}-2\nu\nu^\top$ is the reflection across the hyperplane
$x_i=x_{i+1}$. The face condition in \eqref{eq:faces} on $y_i=0$ is
$u_{x_i}=u_{x_{i+1}}$, that is, $\nu\cdot\nabla\bar u=0$. It was verified
on the relatively open face of $\mathcal C$ in this hyperplane, where all the other
inequalities are strict, and it holds on the closed face by continuity of
$\nabla\bar u$. The relatively open face is an open subset of the
hyperplane, so the identity may be differentiated there in every direction
orthogonal to $\nu$, which shows that $\nabla^2\bar u\,\nu$ is orthogonal to the hyperplane, that is, parallel to $\nu$; this again
extends to the closed face by continuity. Thus at $x$ the vector $\nu$ is
orthogonal to $\nabla\bar u(x)$ and is an eigenvector of
$\nabla^2\bar u(x)$, say $\nabla^2\bar u(x)\nu=\gamma\nu$. Hence
$\sigma_i\nabla\bar u(x)=\nabla\bar u(x)-2\bigl(\nu\cdot\nabla\bar u(x)\bigr)\nu
=\nabla\bar u(x)$ and
\[
 \sigma_i\nabla^2\bar u(x)\sigma_i
 =\nabla^2\bar u(x)-2\nu\nu^\top\nabla^2\bar u(x)-2\nabla^2\bar u(x)\nu\nu^\top
  +4\nu\nu^\top\nabla^2\bar u(x)\nu\nu^\top=\nabla^2\bar u(x),
\]
since the last three terms equal $-2\gamma\nu\nu^\top$, $-2\gamma\nu\nu^\top$,
and $4\gamma\nu\nu^\top$. This verifies the agreement hypothesis, and
Lemma~\ref{lem:gluing} gives $u\in C^2(\R^5)$. This completes the proof
of Proposition~\ref{prop:C2}.

\subsection{The Hamiltonian inequalities in Region III}\label{sec:H3}

It remains to prove $D_\v^2u\leq D_{\v_*}^2u$ for every binary control.
We treat Region III first, because the finite formula \eqref{eq:finite3}
is affine in each $\tanh a_i$, $i\leq3$, so the inequalities there reduce
by algebra alone to finitely many scalar inequalities in $L$. In Regions I
and II the fifth expert enters \eqref{eq:formula12} through the factor
$\eu^{-2z_4\coth t}$ inside the integral, and each gap is a Laplace transform
in $2z_4$; the reduction to scalar inequalities then needs the sign
principles of Section~\ref{sec:H12}. 

Translation invariance gives $\nabla^2u\one=0$, so complementary controls have equal curvature;
we use the sixteen representatives with $v_1=0$.
In the coordinates \eqref{eq:coords3}, define
\begin{equation}\label{eq:q3}
 q_\v=\mathcal J_{\mathrm{III}}b_\v,\qquad\mathcal J_{\mathrm{III}}=\frac13
 \begin{pmatrix}1&0&-1&-2\\1&0&-1&1\\1&0&2&1\\1&3&2&1\end{pmatrix},
\end{equation}
where $\mathcal J_{\mathrm{III}}$ is the Jacobian $\partial a/\partial y$ of
\eqref{eq:coords3}, so that $b_\v^\top\nabla_y^2Fb_\v=q_\v^\top\nabla_a^2Fq_\v$
in \eqref{eq:direction-scaling}; the direction $b_{\v_*}=(1,-1,1,0)$ is
sent to $(0,0,1,0)$, and $D_{\v_*}^2u=(2/k)F$ by \eqref{eq:fixed}. The
physical curvature gap of $\v$, $D_{\v_*}^2u-D_\v^2u$, is therefore
$(2/k)(F-q_\v^\top\nabla_a^2Fq_\v)$.
Every term of \eqref{eq:finite3} has degree one in each pair
$(\cosh a_i,\sinh a_i)$, $i\leq3$, and differentiation preserves this, so
after division by $\prod_{i=1}^3\cosh a_i$ the gap is affine separately
in each of $\tanh a_i$. Recall that $L$ denotes $a_4$ in Region III, and
put $\xi_i=\tanh a_i$ for $i\leq3$. The closed Region III is then the simplex
$0\leq\xi_1\leq\xi_2\leq\xi_3\leq\tanh L$, and the normalized gap is
affine in each $\xi_i$ separately. A function affine in each variable
separately on the cube $[0,\tanh L]^3$ is a convex combination of its
values at the eight vertices of the cube, so it suffices to prove that the
normalized gap is nonnegative at these vertices, where every
$a_i\in\{0,L\}$. Permutations of $a_1,a_2,a_3$ leave \eqref{eq:finite3}
invariant and permute $v_3,v_4,v_5$, since they correspond to permutations
of experts 3, 4, and 5; they map the gap for one control at a vertex to
the gap for the permuted control at the permuted vertex. It therefore
suffices to consider, for all sixteen controls, the four vertices
\[
 \omega_0=(0,0,0),\qquad \omega_1=(0,0,L),\qquad \omega_2=(0,L,L),\qquad
 \omega_3=(L,L,L)
\]
of the simplex, where $\omega_j$ has $j$ coordinates equal to $L$. At
$(a_1,a_2,a_3)=\omega_j$ the normalized gap is a function of $a_4=L$
alone, through the coefficients \eqref{eq:coefficients} and the
derivatives in the direction $a_4$; we write
\begin{equation}\label{eq:Gj}
 G_j^\v(L)=\left.
 \frac{F-q_\v^\top\nabla_a^2Fq_\v}{\cosh a_1\cosh a_2\cosh a_3}
 \right|_{(a_1,a_2,a_3)=\omega_j},\qquad j=0,1,2,3.
\end{equation}

For $j=0$, we obtain the following nonzero expressions. The gap is zero
for every control omitted from the table.
\begin{center}
\begin{tabular}{ll}\toprule
Controls & $G_0^\v$\\\midrule
$00000$ & $e$\\
$00001,00010,00100$ & $e/3+2\eta/21+2S\lambda/7$\\
$00011,00101,00110$ & $(3\eta+2S\lambda)/7$\\
$00111$ & $\eta$\\
$01000,01111$ & $(7e+5\eta-6S\lambda)/21$\\\bottomrule
\end{tabular}
\end{center}
Only $U=7e+5\eta-6S\lambda$ needs more than positivity of the displayed
factors. Differentiate with \eqref{eq:e-first} for $e'$, \eqref{eq:eta}
for $\eta'$, and $(S\lambda)'=C\lambda-1$; then eliminate $I$ through
$\eta=3\coth L-15I/\sinh^6L$ and $\eta$ through
$p=\tfrac12\tanh L-\eta$. This gives
\[
 U'-\tanh L\,U
 =\tfrac95\sech^2L-6\sech L\,\lambda
   +(30\coth L+\tfrac{18}5\tanh L)p.
\]
Using the upper bound $p\leq\tanh L\sech^2L/8$ of Lemma~\ref{lem:p-bounds}
in the last term and $\lambda=\atanh(\sech L)\geq\sech L$ in the second,
\[
 U'-\tanh L\,U\leq
 \bigl(\tfrac95+\tfrac{15}4-6\bigr)\sech^2L+\tfrac9{20}\tanh^2L\,\sech^2L
 =-\tfrac9{20}\sech^4L<0,
\]
since $\tfrac95+\tfrac{15}4-6=-\tfrac9{20}$ and $\tanh^2L=1-\sech^2L$.
Since $U/\cosh L\to0$, integration from infinity proves
$U\geq(9/20)\cosh L\int_L^\infty\sech^5t\dd>0$.

For the other vertices define eight generators:
\begin{equation}\label{eq:eight-generators}
\begin{aligned}
 C_1&=G_1^{00011},&D_1&=G_1^{00101},&T_1&=G_1^{01111},\\
 A_2&=G_2^{00011},&B_2&=G_2^{00110},&T_2&=G_2^{01111},\\
 P_3&=G_3^{00011},&&&T_3&=G_3^{01111}.
\end{aligned}
\end{equation}
The following table lists all controls. We obtain each entry by differentiating \eqref{eq:finite3} and using
\eqref{eq:e-ode}; all entries are exact identities.
\begin{center}\small
\begin{tabular}{llll}\toprule
$\v$ & $G_1^\v$ & $G_2^\v$ & $G_3^\v$\\\midrule
00000 & $C_1+3T_1+4D_1$ & $3(A_2+B_2+T_2)$ & $3T_3+6P_3$\\
00001 & $C_1+T_1+2D_1$ & $3A_2+T_2$ & $T_3+3P_3$\\
00010 & $C_1+T_1+2D_1$ & $A_2+2B_2+T_2$ & $T_3+3P_3$\\
00011 & $C_1$ & $A_2$ & $P_3$\\
00100 & $T_1+2D_1$ & $A_2+2B_2+T_2$ & $T_3+3P_3$\\
00101 & $D_1$ & $A_2$ & $P_3$\\
00110 & $D_1$ & $B_2$ & $P_3$\\
00111 & $0$ & $0$ & $0$\\
01000 & $T_1+2D_1$ & $A_2+2B_2+T_2$ & $T_3+3P_3$\\
01001 & $D_1$ & $A_2$ & $P_3$\\
01010 & $D_1$ & $B_2$ & $P_3$\\
01011 & $0$ & $0$ & $0$\\
01100 & $0$ & $B_2$ & $P_3$\\
01101 & $0$ & $0$ & $0$\\
01110 & $0$ & $0$ & $0$\\
01111 & $T_1$ & $T_2$ & $T_3$\\\bottomrule
\end{tabular}
\end{center}
It therefore suffices to prove nonnegativity of the eight generators.

Substitute \eqref{eq:e-first} into a generator and write $G=a(L)e(L)+b(L)$.
The coefficient $a$ is rational in $\varrho=\eu^L$; $b$ is rational in $\varrho$ and
linear in $L,\lambda$. For $h=a\cosh L$ the trace cancels from
\begin{equation}\label{eq:first-scalar}
 \mathcal R=\left(\frac Gh\right)'
   =\left(\frac b{a\cosh L}\right)'
     -\frac{3I}{\cosh^2L\sinh^5L},
\end{equation}
because $G/h=e/\cosh L+b/(a\cosh L)$ and, by \eqref{eq:e-first},
\[
 \Bigl(\frac e{\cosh L}\Bigr)'=\frac{e'-\tanh L\,e}{\cosh L}
 =-\frac{3I}{\cosh^2L\sinh^5L}.
\]
For six generators, $a$ has a fixed sign and the
certified sign of $\mathcal R$ is the opposite sign:
\begin{center}\small
\begin{tabular}{lll}\toprule
Generator & $a(L)$, with $\varrho=\eu^L$ & Sign of $\mathcal R$\\\midrule
$C_1$ & $-5(\varrho^2-1)^2/[8(\varrho^2+1)^2]$ & $\mathcal R\geq0$\\
$D_1$ & $5(\varrho^2-1)^2/[12(\varrho^2+1)^2]$ & $\mathcal R\leq0$\\
$T_1$ & $(\varrho^4+30\varrho^2+1)/[24(\varrho^2+1)^2]$ & $\mathcal R\leq0$\\
$A_2$ & $5(\varrho^2-1)^2/[24(\varrho^2+1)^2]$ & $\mathcal R\leq0$\\
$B_2$ & $5(\varrho^2-1)^2/[6(\varrho^2+1)^2]$ & $\mathcal R\leq0$\\
$P_3$ & $5(\varrho^2-1)^2(11\varrho^4+18\varrho^2+11)/[48(\varrho^2+1)^4]$ & $\mathcal R\leq0$\\\bottomrule
\end{tabular}
\end{center}
The exact sign procedure is given in Section~\ref{sec:certificates}.
Here $b$ grows at most linearly at infinity (in fact it has a finite
limit, as the audit checks), $a$ tends to a nonzero
constant, and $e\to1/2$, so $G/h\to0$. Integrating the certified
derivative sign from infinity proves $G\geq0$ in all six cases: if $a>0$
and $\mathcal R\leq0$, then $G/h$ decreases to zero and is therefore nonnegative;
if $a<0$ and $\mathcal R\geq0$, then $G/h$ increases to zero and is nonpositive,
and $h<0$ again gives $G\geq0$.

For $T_j=\alpha_je+\beta_j$, $j=2,3$, the coefficient $\alpha_j$ changes sign. Instead use
\begin{equation}\label{eq:second-scalar3}
 \mathcal R_j=T_j''-\frac{h_j''}{h_j'}T_j'\geq0,
 \qquad h_j=\alpha_j\cosh L.
\end{equation}
This operator also eliminates the trace: writing $T_j=h_jE+\beta_j$ with
$E=e/\cosh L$, the terms containing $E$ itself cancel in $\mathcal R_j$, and only
$E'=-3I/(\cosh^2L\sinh^5L)$ and $E''$ remain, both explicit by
\eqref{eq:e-first}. The same device is used in \eqref{eq:second-scalar12}
below.
The exact derivatives are
\begin{align*}
 h_2'&=-\frac{(\varrho^2-1)(13\varrho^4+110\varrho^2+13)}{48\varrho(\varrho^2+1)^2}<0,\\
 h_3'&=-\frac{(\varrho^2-1)(77\varrho^8+716\varrho^6+846\varrho^4+716\varrho^2+77)}
                  {96\varrho(\varrho^2+1)^4}<0.
\end{align*}
Thus the operator has no interior singularity. The explicit formulas give
$T_j\to0$, $T_j'$ bounded, and $h_j'\to-\infty$.
Since $(T_j'/h_j')'=\mathcal R_j/h_j'\leq0$ and $T_j'/h_j'\to0$,
we obtain $T_j'/h_j'\geq0$, hence $T_j'\leq0$ and $T_j\geq0$.
This proves every Region III comparison once the eight scalar signs are
certified.

\subsection{Regions I and II: corner gaps and sign principles}\label{sec:H12}

For these regions let $q=\mathcal Jb_\v=(q_0,q_1,q_2,q_3)$ in the coordinates $(z_1,d,z_2,z_4)$ (we suppress the subscript $\v$ of \eqref{eq:q3}),
where $\mathcal J$ is the Jacobian $\partial(z_1,d,z_2,z_4)/\partial y$ of the
regional coordinates, transposed from \eqref{eq:chain-12}:
\begin{equation}\label{eq:q12}
 \mathcal J_{\mathrm I}=\begin{pmatrix}
 1/2&1&1/2&0\\-1/2&0&1/2&0\\1/2&0&1/2&0\\0&0&0&1
 \end{pmatrix},\qquad
 \mathcal J_{\mathrm{II}}=\begin{pmatrix}
 1/2&1&1/2&0\\1/2&0&-1/2&0\\1/2&0&1/2&0\\-1/2&0&1/2&1
 \end{pmatrix}.
\end{equation}
Write $\mathcal G_\v=F-q^\top\nabla^2Fq$, with $\nabla^2$ the Hessian in $(z_1,d,z_2,z_4)$. After division by $\cosh d\cosh z_2$,
this is multiaffine in $\tanh d$ and $\tanh z_2$, including the
moving-endpoint term, which is bilinear in those variables. Enlarge the domain to the
square $0\leq d,z_2\leq z_1$. In Region II the formula is symmetric in
$d,z_2$, and interchanging experts 3 and 4 gives the corresponding control
permutation. In Region I, interchanging experts 1 and 2 instead gives
\begin{equation}\label{eq:missing-corner}
 \mathcal G_\v|_{(d,z_2)=(z_1,0)}-\mathcal G_{\v'}|_{(d,z_2)=(0,z_1)}
 =\sinh z_1\,[1-(v_1-v_2)^2]\geq0,
\end{equation}
where $\v'$ is $\v$ with $v_1,v_2$ interchanged. Indeed, this interchange
reverses the sign of $y_1$, which exchanges $d$ and $z_2$ and sends $b_\v$
to $b_{\v'}$, so the part of $F$ symmetric in $(d,z_2)$, namely the
integral and all of $F_4^\mathrm I$ except
$F^{\mathrm{as}}=-\tfrac12\sinh(z_2-d)$, contributes equally to the two
sides. For the antisymmetric part,
$q^\top\nabla^2F^{\mathrm{as}}q=-\tfrac12\sinh(z_2-d)(q_2-q_1)^2$ with
$q_2-q_1=v_1-v_2$ by \eqref{eq:q12}, so
$F^{\mathrm{as}}-q^\top\nabla^2F^{\mathrm{as}}q
=-\tfrac12\sinh(z_2-d)[1-(v_1-v_2)^2]$, which equals
$\tfrac12\sinh z_1[1-(v_1-v_2)^2]$ at $(d,z_2)=(z_1,0)$ and its negative
at $(0,z_1)$.
Consequently only $(d,z_2)=(0,0),(0,z_1),(z_1,z_1)$ need be checked.
The first two are shared by Regions I and II. The four distinct families are
\begin{center}
\begin{tabular}{clc}\toprule
Family & Scaled gaps $(y_1,y_2,y_3,y_4)$ & $n$\\\midrule
$V_0$ & $(0,z_1,0,z_4)$ & 0\\
$V_1$ & $(z_1,0,z_1,z_4)$ & 1\\
$V_2^{\mathrm I}$ & $(0,0,2z_1,z_4)$ & 2\\
$V_2^{\mathrm{II}}$ & $(2z_1,0,0,z_4+z_1)$ & 2\\\bottomrule
\end{tabular}
\end{center}
The integer $n$ counts how many of $d,z_2$ equal $z_1$.

Put $c_L=\coth L$, $s=2z_4$, and change variables $c=\coth t$ in the
integral, so that $\eu^{-2z_4\coth t}=\eu^{-sc}$, $\diff t=-\diff c/(c^2-1)$, and
$\sinh t=(c^2-1)^{-1/2}$. Its positive density is
\begin{equation}\label{eq:mu}
 \mu(c)=\frac{p(\arccoth c)}{(c^2-1)^{3/2}},\qquad1<c\leq c_L,
\end{equation}
that is, $\int_L^\infty\sinh t\,p(t)f(\coth t)\dd=\int_1^{c_L}\mu(c)f(c)\,\diff c$ for
every integrable $f$, the reversed limits absorbing the sign of $\diff t$.
Each corner gap has the form
\begin{equation}\label{eq:gap-measure}
 \mathcal G_\v(s)=B_\v(L)+\int_1^{c_L}\mu(c)K_\v(L,c)\eu^{-sc}\,\diff c
                         +a_\v(L)\eu^{-sc_L}.
\end{equation}
For the four-expert contribution, we use the two facts from
\cite[Theorems~3.1 and~3.2]{BEZ20} recorded after \eqref{eq:four-coefficients}:
the function $u_4(x_1,\dots,x_4)=x_1+F_4/k$ is the global $C^2$ solution of
the four-expert equation, and $(1,0,1,0)$ attains its Hamiltonian maximum
throughout the closed ordered sector. The identity
$\partial_{z_2}^2F_4=F_4$ identifies the stated maximizing control.
The inequality extends to ties by the $C^2$ regularity in
\cite[Theorem~3.1]{BEZ20}.

For $\bar \v=(v_1,\dots,v_4)$ the four-expert contribution is therefore
\[
 B_\v=F_4-q^\top\nabla^2F_4q
     =\frac{k}{2}\bigl[2(u_4-x_1)-D_{\bar \v}^2u_4\bigr]\geq0.
\]
The fifth component of $\v$ has no effect. Every corner retained above is
in the closed four-expert sector, including $V_2^{\mathrm I}$ and
$V_2^{\mathrm{II}}$. The nonphysical corner $(d,z_2)=(z_1,0)$ is handled
by symmetry and \eqref{eq:missing-corner}; no four-expert inequality outside
the ordered sector is used. The atom comes from the moving lower limit $z_1$ of the integral: as in
\eqref{eq:jet-recipe}, the first $z_1$-derivative produces no endpoint
term, and the second produces
$-\partial_{z_1}(\text{integrand})|_{t=z_1}=p(L)\eu^{-sc_L}\Phi_L(d)\Phi_L(z_2)$,
which vanishes unless $d=z_2=0$ because $\Phi_L(L)=0$. In
$\mathcal G_\v=F-q^\top\nabla^2Fq$ it carries the coefficient $-q_0^2$,
so $a_\v=-q_0^2p(L)$ when $n=0$ and $a_\v=0$ otherwise. It must be
retained, since the second derivative moves the integration endpoint.

We can express all of the kernels using one polynomial formula. With
$d,z_2\in\{0,L\}$ at the corners, write $(\psi_0,\psi_1,\psi_2)=(\Phi_t(L),\Phi_t(d),\Phi_t(z_2))$ and $(N_0,N_1,N_2)=(N_t(L),N_t(d),N_t(z_2))$ in the notation of \eqref{eq:phi} and Section~\ref{sec:modes},
indexed like the components $q_0,q_1,q_2$ of $q$. Since $d,z_2\leq L\leq t$, each $\psi_i$ and $N_i$ is affine in $c$; for example $\psi_1=\cosh d-c\sinh d$ and $N_1=c\cosh d-\sinh d$.
With $P=\psi_0\psi_1\psi_2$, for fixed $c$ the integrand of the correction is $\eu^{-sc}P$. Each $\psi_i$ depends only on its own variable,
and with a prime for the derivative in that variable
\[
 \psi_i'=-N_i,\qquad \psi_i''=\psi_i,\qquad
 \partial_{z_4}\eu^{-sc}=-2c\,\eu^{-sc},\qquad \partial_{z_4}^2\eu^{-sc}=4c^2\eu^{-sc}.
\]
Expanding $q^\top\nabla^2(\eu^{-sc}P)q$ with these rules and subtracting it
from $\eu^{-sc}P$ gives
\begin{equation}\label{eq:kernel-poly}
\begin{aligned}
 K_\v={}&\left(1-\sum_{i=0}^2q_i^2-4q_3^2c^2\right)P
       -2\sum_{i<j}q_iq_jN_iN_j\psi_{3-i-j}-4q_3c\sum_iq_iN_i\prod_{j\neq i}\psi_j.
\end{aligned}
\end{equation}
We use this formula to generate each
of the 64 comparisons.

We next give three elementary sign criteria that will be used in the
Hamiltonian comparisons. We state them for an integrable density $h$ on $[1,c_L]$.

\begin{lemma}\label{lem:laplace}
Let $c_L>1$ and, for $s\geq0$, let $G(s)=B+\int_1^{c_L} h(c)\eu^{-sc}\,\diff c+a\eu^{-sc_L}$, with $B\geq0$,
$a\leq0$, and $G(0)\geq0$.
\begin{enumerate}
\item[(i)] If $h\leq0$, then $G(s)\geq G(0)$. If $h\geq0$, then
$G(s)\geq(1-\eu^{-sc_L})B+\eu^{-sc_L}G(0)$.
\item[(ii)] If $a=0$ and $h$ changes sign at most once, from positive to negative,
then $G(s)\geq0$.
\item[(iii)] If $a=0$, $h$ changes sign at most once from negative to positive,
and $-\int_1^{c_L} c h(c)\,\diff c\geq0$, then $G'(s)\geq0$ and $G(s)\geq0$.
\end{enumerate}
\end{lemma}
\begin{proof}
For the first assertion, if $h\leq0$ then
\[
 G(s)-G(0)=\int_1^{c_L}h(c)\bigl(\eu^{-sc}-1\bigr)\,\diff c+a\bigl(\eu^{-sc_L}-1\bigr)\geq0,
\]
both integrand and atom being products of two nonpositive factors; if
$h\geq0$ then $\eu^{-sc}\geq \eu^{-sc_L}$ on $[1,c_L]$ gives
\[
 G(s)\geq B+\eu^{-sc_L}\Bigl(\int_1^{c_L}h\,\diff c+a\Bigr)=B+\eu^{-sc_L}\bigl(G(0)-B\bigr).
\]
For the second assertion let $c_*$ be the transition point. Then
$h(c)(\eu^{-sc}-\eu^{-sc_*})\geq0$ for every $c$, since both factors change
sign at $c_*$ in the same direction, so with $a=0$
\[
 G(s)\geq B+\eu^{-sc_*}\int_1^{c_L}h\,\diff c=(1-\eu^{-sc_*})B+\eu^{-sc_*}G(0)\geq0.
\]
In the last case, multiplication by $c>0$ preserves the negative-to-positive
ordering of $h$, so $ch(c)(\eu^{-sc}-\eu^{-sc_*})\leq0$ and
\[
 G'(s)=-\int_1^{c_L}ch(c)\eu^{-sc}\,\diff c\geq-\eu^{-sc_*}\int_1^{c_L}ch(c)\,\diff c\geq0,
\]
whence $G(s)\geq G(0)\geq0$.
Zero densities and absent sign changes follow by the same comparisons.
\end{proof}

For the remaining comparisons, we use a second cumulative integral.

\begin{lemma}\label{lem:cumulative}
For the same $G$, define
\[
 \Gamma_2(b)=Bb+\int_1^b(b-c)h(c)\,\diff c\quad(1\leq b\leq c_L),\qquad
 \Gamma_2(b)=Bb\quad(0\leq b\leq1).
\]
If $G(0)\geq0$ and $\Gamma_2\geq0$ on $[0,c_L]$, then $G(s)\geq0$ for $s\geq0$.
\end{lemma}
\begin{proof}
Let $\Gamma_1(b)=B+\int_1^bh\,\diff c$ for $b\geq1$ and $\Gamma_1(b)=B$ for
$0\leq b\leq1$, so that $\Gamma_2'=\Gamma_1$, $\Gamma_2(0)=0$, and
$\Gamma_1(c_L)+a=G(0)$. Integrating by parts once,
\[
 G(s)=B+\int_0^{c_L}\eu^{-sc}\Gamma_1'(c)\,\diff c+a\eu^{-sc_L}
     =\eu^{-sc_L}\bigl(\Gamma_1(c_L)+a\bigr)+s\int_0^{c_L}\eu^{-sc}\Gamma_1(c)\,\diff c,
\]
and integrating by parts once more, with $\Gamma_1=\Gamma_2'$, gives the exact
identity
\begin{equation}\label{eq:cumulative-identity}
 G(s)=\eu^{-sc_L}G(0)+s \eu^{-sc_L}\Gamma_2(c_L)
                    +s^2\int_0^{c_L} \eu^{-sb}\Gamma_2(b)\,\diff b.
\end{equation}
The atom at $c_L$ enters $G(0)$, and its contribution to $\Gamma_2(c_L)$ is zero.
All terms on the right are nonnegative.
\end{proof}

We can now give the full classification. The symbols $+$, $-$, $0$ denote the
sign of the polynomial kernel, not the sign or vanishing of the full gap;
PN denotes a positive-to-negative transition,
M uses the moment condition in Lemma~\ref{lem:laplace}, and C uses
Lemma~\ref{lem:cumulative}. M and C each reduce to three distinct gaps,
listed below; the audit identifies these cases by exact equality of the
triple (base, kernel, atom). The standing hypotheses $B\geq0$ and
$\mathcal G_\v(0)\geq0$ of Lemmas~\ref{lem:laplace}
and~\ref{lem:cumulative} are supplied for all $64$ cases in
Section~\ref{sec:H122}, through the boundary values at $z_4=0$ and
\eqref{eq:bottom-cone}.
\begin{center}
\begin{tabular}{ccccc}\toprule
$\v$ & $V_0$ & $V_1$ & $V_2^{\mathrm I}$ & $V_2^{\mathrm{II}}$\\\midrule
00000 & $+$ & $+$ & $+$ & $+$\\
00001 & $-$ & $-$ & $-$ & $-$\\
00010 & $+$ & PN & PN & $+$\\
00011 & $-$ & $-$ & $-$ & M\\
00100 & $+$ & $+$ & $+$ & $+$\\
00101 & $-$ & M & M & M\\
00110 & $+$ & $+$ & $+$ & $+$\\
00111 & 0 & 0 & 0 & 0\\
01000 & $+$ & $+$ & $+$ & $+$\\
01001 & C & M & M & M\\
01010 & 0 & $+$ & $+$ & $+$\\
01011 & 0 & 0 & 0 & 0\\
01100 & 0 & 0 & 0 & $+$\\
01101 & 0 & 0 & $+$ & 0\\
01110 & C & C & M & C\\
01111 & $+$ & PN & $+$ & PN\\\bottomrule
\end{tabular}
\end{center}
We emphasize that the base and endpoint atom in \eqref{eq:gap-measure}
are still included in each comparison. For example, at $V_0$ with $\v=00111$, the table records a zero
kernel, but
\[
 \mathcal G_\v(s)=\tfrac12\tanh L-p(L)\eu^{-sc_L},\qquad
 \mathcal G_\v(0)=\eta(L)>0.
\]
Lemma~\ref{lem:laplace} accounts for the negative atom, and the gap is
positive for every $L>0$ and $s\geq0$.
We verify the polynomial sign classifications as follows. We divide $K_\v$ by
$S^{n+1}$, write it as a polynomial in $c_L,c$, and substitute
$c_L=1+r$, $c=1+rw$, where $r>0$ and $0\leq w\leq1$.
Its Bernstein coefficients in $w$ are polynomials in $r$. The signs of their
power coefficients establish fixed signs or at most one sign variation.
For a sequence with one uncertain coefficient, either of its signs gives
the same allowable variation bound. Bernstein variation reduction follows
by substituting $w=t/(1+t)$ and applying Descartes' rule to the resulting
power polynomial.
There is one exceptional M kernel, $V_2^{\mathrm I}$ with $\v=00101$.
Its Bernstein coefficients are
\begin{gather*}
 -\frac54r^3,\quad -\frac{r^3(6r-7)}{10},\quad
 -\frac{3r^3(2r^2-8r-17)}{40}, \quad \frac{r^3(r+2)(3r+5)}{10},\quad\frac{r^3(r+2)^2}{10},\quad0.
\end{gather*}
The two uncertain coefficients cannot have signs $+,-$, since
$2r^2-8r-17<0$ on $0\leq r\leq7/6$. Thus this kernel also has at most one
negative-to-positive transition. The exact audit reconstructs each kernel
from \eqref{eq:kernel-poly} before performing these sign checks.

\subsection{Regions I and II: the scalar inequalities}\label{sec:H122}

At $z_4=0$, the families $V_0,V_1,V_2^{\mathrm{II}}$ coincide with the Region III
vertices with respectively zero, one, and two entries equal to $L$.
The $C^2$ matching therefore gives all 48 boundary inequalities there.
For the remaining family $V_2^{\mathrm I}$, the comparisons at $z_4=0$
reduce to the three generators
\[
 C_*=\mathcal G_{00011},\qquad
 T_*=\mathcal G_{01111},\qquad D_*=\mathcal G_{00110}.
\]
The kernels of $T_*,D_*$ are nonnegative, these being the $+$ entries at
$(V_2^{\mathrm I},01111)$ and $(V_2^{\mathrm I},00110)$ of the table, and
their bases are nonnegative by the four-expert inequality. The full
remaining list is
\begin{equation}\label{eq:bottom-cone}
\begin{aligned}
 \mathcal G_{00001}=\mathcal G_{00010}&=C_*+T_*+D_*,\\
 \mathcal G_{00100}=\mathcal G_{01000}&=T_*,\\
 \mathcal G_{00101}=\mathcal G_{01001}=\mathcal G_{01110}
   =\mathcal G_{01010}=\mathcal G_{01101}&=D_*,\\
 \mathcal G_{00111}=\mathcal G_{01011}=\mathcal G_{01100}&=0.
\end{aligned}
\end{equation}
Control $00000$ is harmless because $F\geq0$. Only $C_*\geq0$ needs an
additional scalar sign; we certify $C_*/\cosh^2L$.

All needed scalar integrals reduce to seven moments:
\begin{equation}\label{eq:moments-def}
 M_j(L)=\int_1^{c_L}c^j\mu(c)\,\diff c
       =\int_L^\infty\sinh t\,p(t)\coth^jt\dd,
 \qquad0\leq j\leq6.
\end{equation}
Write $p$ for $p(L)$; the first two are
\begin{equation}\label{eq:moments01}
 M_1=\frac{Sp+1/(2C)}5,
 \qquad M_0=\frac{e-C\theta+SM_1}{C}.
\end{equation}
The second identity is \eqref{eq:g-green} at $X=L$, since
$\sinh(t-L)=\sinh t\cosh L-\cosh t\sinh L$ gives $g(L)=CM_0-SM_1$
and $g=e-e_4=e-C\theta$. The first follows from the density equation
\eqref{eq:p-ode}, which gives
$(\sinh t\,p)'=\cosh t\,p+\sinh t\,p'=-5\cosh t\,p+\tfrac12\tanh t\,\sech t$;
integrating from $L$ to infinity, where $\sinh t\,p\to0$, yields
$5M_1=Sp+\tfrac12\sech L$. In the $c$ variable, with
$c'=-(c^2-1)$ and $\sech^2t=(c^2-1)/c^2$, the density equation reads
\[
 \mu'-\frac{3c}{c^2-1}\mu=-\frac1{2c^2(c^2-1)^{3/2}}.
\]
Hence
\[
 \frac{\diff}{\diff c}\bigl[c^j(c^2-1)\mu\bigr]
 =(j+5)c^{j+1}\mu-jc^{j-1}\mu-\frac{c^{j-2}}{2\sqrt{c^2-1}},
\]
and integrating from $1$ to $c_L$, with $(c_L^2-1)\mu(c_L)=Sp$, gives
\begin{equation}\label{eq:momentrec}
 M_{j+1}=\frac{c_L^jSp+jM_{j-1}+Q_{j-2}/2}{j+5},\qquad j\geq1,
\end{equation}
where $Q_r=\int_1^{c_L} c^r/\sqrt{c^2-1}\,\diff c=\int_L^\infty\coth^rt\,\csch t\dd$
and the needed values are
\begin{equation}\label{eq:Q-primitives}
 Q_{-1}=2\theta,\quad Q_0=\lambda,\quad Q_1=S^{-1},\quad
 Q_2=\tfrac12(C/S^2+\lambda),\quad Q_3=\tfrac1{3S^3}+S^{-1}.
\end{equation}
The boundary term at $c=1$ is zero, since
$(c^2-1)\mu=\O(\sqrt{c-1})$. These formulas are also checked by
$M_j'=-Sp\,c_L^j$ and their zero limits at infinity.

There are only three distinct M comparisons: $(V_1,00101)$,
$(V_2^{\mathrm I},00101)$, and $(V_2^{\mathrm{II}},00011)$. For these the
moment condition of Lemma~\ref{lem:laplace}(iii), with $h=\mu K_\v$, is
$-\int_1^{c_L}c\,\mu(c)K_\v(L,c)\,\diff c\geq0$; if $K_\v=\sum_jk_j(L)c^j$ for the comparison $(V,\v)$,
this reads
\begin{equation}\label{eq:N-moments}
 N_{V,\v}(L)=-\cosh^{-n}L\,\sum_j k_j(L)M_{j+1}(L)\geq0.
\end{equation}
The factor $\cosh^{-n}L>0$ is only a convenient normalization.
The moment recursion makes all three finite expressions, linear in
$e,p,\theta,\lambda$; the certificates below establish their signs.

We next treat the three cumulative comparisons, marked C in the sign
table.
The distinct C comparisons are $(V_0,01001)$, $(V_1,01110)$, and
$(V_2^{\mathrm{II}},01110)$. All have $\mathcal G_\v(0)=0$.
For $t\geq L$ define
\begin{equation}\label{eq:J-cumulative}
 J(L,t)=B_\v(L)+\int_1^{\coth t}(1-c\tanh t)\mu(c)K_\v(L,c)\,\diff c.
\end{equation}
This is $\Gamma_2(\coth t)/\coth t$. Let $m=n+1$, which takes the values
$1,2,3$ in these three cases, and put
\begin{equation}\label{eq:H-cumulative}
 x=\eu^{-2L},\quad y=\eu^{-2t},\qquad
 H(x,y)=\eu^{-mL}J(L,t),\qquad0<y\leq x<1.
\end{equation}
For fixed $c$, $\eu^{-mL}K_\v(L,c)$ is a polynomial of degree at most $m$ in $x$.
Thus its $(m+1)$st derivative vanishes, and only the four-expert base
contributes to the next derivative. We obtain the positive rational function
\begin{equation}\label{eq:positive-forcing}
 \partial_x^{m+1}H(x,y)
 =\frac{m!\varrho^{2m+5}}{(\varrho^2-1)^m(\varrho^2+1)^3}>0,
 \qquad \varrho=\eu^L=x^{-1/2}.
\end{equation}
This identity follows by differentiating \eqref{eq:four-background} with
the specified control; it is checked exactly for all three cases.

It remains to prove the nine diagonal signs
\begin{equation}\label{eq:Aj}
 \Lambda_j(t)=\left.\partial_x^jH(x,y)\right|_{x=y=\eu^{-2t}}\geq0,
 \qquad 0\leq j\leq m.
\end{equation}
For fixed $y>0$, the kernel part of $H(\cdot,y)$ is polynomial and its
base is smooth for $L>0$, so $H(\cdot,y)\in C^{m+1}([y,x])$ whenever
$y\leq x<1$. Taylor's formula with integral remainder then gives
\begin{equation}\label{eq:positive-Taylor}
 H(x,y)=\sum_{j=0}^m\frac{\Lambda_j(t)}{j!}(x-y)^j
 +\frac1{m!}\int_y^x(x-x')^m\partial_x^{m+1}H(x',y)\,\diff x'\geq0.
\end{equation}
Thus $\Gamma_2\geq0$ and Lemma~\ref{lem:cumulative} applies.

To write the scalar functions explicitly, we set
\[
 \partial_x^j(\eu^{-mL}K_\v)=\sum_r k_{jr}(L)c^r,
 \qquad
 \partial_x=-\tfrac12\eu^{2L}\partial_L.
\]
Differentiate with $c$ fixed before taking the diagonal $t=L$. Then
\begin{equation}\label{eq:Aj-formula}
 \Lambda_j(L)=\partial_x^j(\eu^{-mL}B_\v)+\sum_r k_{jr}(L)
                     \bigl[M_r(L)-\tanh L\,M_{r+1}(L)\bigr].
\end{equation}
Equations \eqref{eq:kernel-poly}, \eqref{eq:moments01},
\eqref{eq:momentrec}, and \eqref{eq:Aj-formula} specify all nine scalars
by finite rational operations and the trace. We use these exact expressions as inputs to the certificate calculation.

We now prove the thirteen remaining scalar inequalities.
Substitute \eqref{eq:p-elementary} into the nine $\Lambda_j$, the three $N_{V,\v}$,
and $C_*/\cosh^2L$. Every $\theta$ term cancels. The resulting expressions have
the same form $G=ae+b$ as in the Region III proof.
Eleven have first-order certificates \eqref{eq:first-scalar}, with
$-\operatorname{sgn}(a)\mathcal R\geq0$:
\begin{center}\small
\begin{tabular}{lll}\toprule
Scalar & Indices & Signs of $a$\\\midrule
$\Lambda_j$ at $(V_0,01001)$ & $j=0,1$ & $+,-$\\
$\Lambda_j$ at $(V_1,01110)$ & $j=0,1,2$ & $+,-,-$\\
$\Lambda_j$ at $(V_2^{\mathrm{II}},01110)$ & $j=0,2,3$ & $+,-,-$\\
$C_*/\cosh^2L$ at $V_2^{\mathrm I}$ & & $-$\\
$N_{V_2^{\mathrm I},00101}$ & & $+$\\
$N_{V_2^{\mathrm{II}},00011}$ & & $-$\\\bottomrule
\end{tabular}
\end{center}
Fixed signs of $a$ are verified by exact polynomial root counts on $\varrho>1$,
after removing zeros at $\varrho=1$, and evaluation at a rational test point.
The endpoint calculation gives $b/(aC)\to0$, while $e/C\to0$.
Integration from infinity proves all eleven scalar inequalities.

The remaining moment at $(V_1,00101)$ is independent of $e$:
\begin{equation}\label{eq:direct-moment}
\begin{aligned}
 N_{V_1,00101}={}&\bigl[105\lambda \varrho^8+6\lambda \varrho^6-6\lambda \varrho^2-105\lambda
       -256p \varrho^5-256p \varrho^3\\
       &\hspace{15mm}-210\varrho^7-82\varrho^5+82\varrho^3+210\varrho\bigr]
       \big/\bigl[1008\varrho(\varrho^2-1)(\varrho^2+1)^2\bigr].
\end{aligned}
\end{equation}
After substitution for $p$, it is certified directly.

Finally $\Lambda_1$ at $(V_2^{\mathrm{II}},01110)$ has the sign-changing coefficient
\[
 a=\frac{3(\varrho^2-1)(7\varrho^4-14\varrho^2-33)}{256\varrho^3(\varrho^2+1)^2}.
\]
Multiply by $\sigma=512\varrho^4(\varrho^2+1)/[3(\varrho^2-1)]>0$ and put
$\widetilde G=\sigma\Lambda_1$. Its homogeneous solution is
$h=\sigma a C=7\varrho^4-14\varrho^2-33$, with $h'=28\varrho^2(\varrho^2-1)>0$.
The exact certificate is
\begin{equation}\label{eq:second-scalar12}
 \widetilde G''-\frac{h''}{h'}\widetilde G'\geq0.
\end{equation}
Expansion \eqref{eq:e-asymptotic} gives
$\widetilde G\to0$ and $\widetilde G'/h'\to0$; only the accuracy
$o(\eu^{-5L})$ of its three displayed terms is needed for these two limits,
since the coefficients of $e$ in $\widetilde G$ and in $\widetilde G'/h'$
are $\O(\varrho^5)$.
Hence $(\widetilde G'/h')'\geq0$, so $\widetilde G'\leq0$ and
$\widetilde G\geq0$. The purpose of the positive rescaling is to ensure that the homogeneous
derivative has no zero in the interior.

\subsection{Exact rational certificates and the supplement}\label{sec:certificates}

We now describe the exact certificates used to verify the scalar
inequalities. The full arithmetic data are included in the supplement. Every signed residual just used can be
written as
\begin{equation}\label{eq:rational-log-form}
 \frac{A(q)(-\log q)+B(q)\,2\atanh q+C(q)}{D(q)},
 \qquad q=\eu^{-L}\in(0,1),\quad D(q)>0,
\end{equation}
where the coefficients are rational polynomials. Derivatives are evaluated
exactly using $\varrho\partial_\varrho$, $L'=1$, $\lambda'=-1/\sinh L$, and
\eqref{eq:e-first}. Polynomial root counts check every denominator sign
and exclude interior singularities.

For $0\leq x<1$ and $N=12$,
\begin{equation}\label{eq:log-bound}
 2\sum_{j=0}^{N-1}\frac{x^{2j+1}}{2j+1}
 \leq2\atanh x\leq
 2\sum_{j=0}^{N-1}\frac{x^{2j+1}}{2j+1}+\frac{2x^{2N+1}}{(2N+1)(1-x^2)}.
\end{equation}
The last inequality bounds the tail by a geometric series. Use the seven
intervals with endpoints
\begin{equation}\label{eq:boxes}
 0,\quad 1/64,\quad1/32,\quad1/16,\quad1/8,\quad1/4,\quad1/2,\quad1.
\end{equation}
On a middle dyadic interval $[2^{-j},2^{1-j}]$ write
\[
 -\log q=j\log2-2\atanh\frac{2^jq-1}{2^jq+1},
 \qquad\log2=2\atanh(1/3).
\]
On $[1/2,1)$ use $-\log q=2\atanh((1-q)/(1+q))$.
On $(0,1/64]$ use $6\log2\leq-\log q\leq(q^{-1}-q)/2$.
All the resulting logarithm bounds are rational functions.

If a coefficient of a logarithm has a fixed sign, substitute its appropriate
upper or lower bound to obtain a lower bound for the entire residual.
There is also a fully explicit rule for a mixed-sign coefficient $A(q)$.
Let $\widehat A(q)$ be the polynomial with absolute values of its power
coefficients, so $|A(q)|\leq\widehat A(q)$ for $q\geq0$. If
$\underline f\leq f\leq\bar f$, then
\begin{equation}\label{eq:mixed-coefficient-bound}
 A(q)f(q)\geq A(q)\underline f(q)-\widehat A(q)\bigl(\bar f(q)-\underline f(q)\bigr).
\end{equation}
Thus coefficient sign changes create no unresolved interval decision.

After clearing denominators whose positivity has been verified, each lower
bound is a rational polynomial $P(q)$. For an interval $[a,b]$, write
\begin{equation}\label{eq:bernstein}
\begin{aligned}
 P(a+(b-a)w)&=\sum_{j=0}^\kappa c_jw^j
            =\sum_{i=0}^\kappa B_i\binom\kappa i w^i(1-w)^{\kappa-i},\\
 B_i&=\sum_{j=0}^i c_j\frac{\binom ij}{\binom\kappa j}.
\end{aligned}
\end{equation}
Every $B_i$ is nonnegative as an exact rational number. Therefore $P\geq0$
on its whole interval. Positive common-denominator rescaling gives the
integer coefficient lists in the supplement.

There are eight scalar certificates in Region III and thirteen in Regions I
and II. Each uses the seven intervals \eqref{eq:boxes}, with no additional
subdivision:
\begin{center}
\begin{tabular}{lrrr}\toprule
Region & Scalar signs & Polynomial certificates & Largest degree\\\midrule
III & 8 & 56 & 79\\
I and II & 13 & 91 & 72\\\midrule
Total & 21 & 147 & 79\\\bottomrule
\end{tabular}
\end{center}
All polynomial power coefficients, rational intervals, integer Bernstein
coefficients, and exact scalar expressions are included in the two certificate
archives. Separate checkers recompute the Bernstein coefficients using only
standard-library fractions; no symbolic library or floating-point calculation
is used in that independent arithmetic check. The identity audits reconstruct
the scalar expressions from the formula and match them to the certificate
inputs. The released verifier independently reconstructs every rational lower
bound from \eqref{eq:log-bound}--\eqref{eq:mixed-coefficient-bound}, checks all
denominator signs, and matches the resulting polynomial to its certificate
before the Bernstein check. This gives a reproducible exact
verification of all the scalar inequalities used above.

\par\bigskip\noindent\begin{minipage}{\linewidth}
The certificates and checks described in this section can be rerun
from the computational supplement, as follows.
The self-contained supplement is the versioned directory
\texttt{five-expert-proof-1.0.6}, archived at
\texttt{doi:10.5281/zenodo.22723924} \cite{Supplement} and distributed as
\texttt{five-expert-proof-1.0.6.zip}, whose SHA-256 checksum is
\begin{center}\small\texttt{a58aae99f888ecdd55fe402027f3928b8fabd90807000bcf0d520ce264e32486}\end{center}
It contains all exact proof inputs, checkers,
and the pinned dependency wheels. From the directory
containing the extracted release, install and run entirely offline:
\begin{verbatim}
python3.12 -m venv proof-env
proof-env/bin/python -m pip install --no-index --require-hashes \
  --find-links five-expert-proof-1.0.6/wheels \
  -r five-expert-proof-1.0.6/requirements.txt
proof-env/bin/python five-expert-proof-1.0.6/verify.py \
  --output proof-results
\end{verbatim}
\end{minipage}\par
The output directory must be new and outside the release. If pip is
configured for user installs, the install step needs the environment
setting \texttt{PIP\_USER=0}. The audit executes
\path{checks/trace_kernels.py}, \path{checks/formula.py},
\path{checks/faces.py}, \path{checks/region3.py}, \path{checks/region12.py},
\path{checks/comb.py}, \path{checks/log_bounds.py}, and
\path{checks/rejection_tests.py}. In particular, it reconstructs every
polynomial bound from its scalar expression before checking its exact
Bernstein coefficients; no numerical sign sampling is used. Adding \texttt{-O}
to the Python command exercises the same checks under optimization.
Logs, dependency versions, exit statuses, and the release manifest hash are
written to the output directory.

The supplement's \path{README.md} gives the installation and audit
procedure, and \path{REVIEW_GUIDE.md} maps every computational check, and the
checker's variable names, to the
paper and states the analytic steps requiring mathematical review.

\par\bigskip\noindent\begin{minipage}{\linewidth}
The verification of this section, the regularity of Section~\ref{sec:reg}, and
Theorem~\ref{thm:comb} are also formalized in Lean~4 on top of Mathlib
\cite{LeanCert}. The release is the versioned directory
\texttt{five-expert-lean-0.1.13}, archived at
\texttt{doi:10.5281/zenodo.23888815} and distributed as
\texttt{five-expert-lean-0.1.13.zip}, whose SHA-256 checksum is
\begin{center}\small\texttt{9c14dff063962d2f0c04f990a66a84741e76bbed355d64ad09dae55ca7d8c32c}\end{center}
The toolchain is pinned by \texttt{lean-toolchain} and Mathlib by
\texttt{lake-manifest.json}. From the directory containing the extracted
release:
\begin{verbatim}
curl -sSf https://elan.lean-lang.org/elan-init.sh | sh
cd five-expert-lean-0.1.13
lake exe cache get
lake build
\end{verbatim}
\end{minipage}
\vspace{1mm}

The build succeeds only if every proof checks and the axiom audit passes on all
$1730$ exported declarations, so a \texttt{sorry} or an axiom introduced by
native evaluation anywhere beneath a result stops it. The formalized chain runs
from the $21$ scalar inequalities to the two theorems: the Hamiltonian
inequalities at every ordered point, $u\in C^2(\mathbb R^5)$, $u-\max_ix_i$
bounded, the expansion at the origin, the optimality set of COMB, and the
four-expert inequality of \cite{BEZ20} on which Regions~I and~II rest. What is
not formalized is the viscosity characterization in Theorem~\ref{thm:main},
which cites the comparison principle of \cite{BEZ20}; $u$ is proved to solve
\eqref{eq:pde} classically on $\mathbb R^5$. Unlike the computational
supplement, the Lean build is not offline, since \texttt{lake exe cache get}
downloads Mathlib's prebuilt objects. The release's \path{README.md} and the
companion note \path{lean-note.pdf}, archived with it, describe the development
and what a reader must check to audit it.

\subsection{Completion of the proofs of the main theorems}\label{sec:comb}

We now complete the proofs of our main results from Section \ref{sec:main}. 

\begin{proof}[Proof of Theorem \ref{thm:main}]
Since $u\in C^2$ is a classical solution, it is also a viscosity solution, and it has linear growth because $u-\phi$ is bounded. By \cite[Proposition~2.2]{BEZ20}, the viscosity solution of \eqref{eq:pde} is unique in the class of functions with linear growth, so $u$ is the unique viscosity solution with bounded correction. We turn to the proof of the expansion \eqref{eq:origin-expansion}. By Proposition~\ref{prop:C2}, $u$ is $C^2$ at the origin. Permutation
symmetry and $u(x+c\one)=u(x)+c$ give $\nabla u(0)=\tfrac15\one$ (recall \eqref{eq:origin}) and they
force the Hessian $\nabla^2u(0)$ to commute with all permutations and to
annihilate $\one$, so $\nabla^2u(0)=\alpha(\mathrm{Id}-\tfrac15\one\one^\top)$ for some
$\alpha$. For a binary $\v$ with $|\v|=\sum_iv_i$ entries equal to one,
$\v^\top\nabla^2u(0)\v=\alpha\,|\v|(5-|\v|)/5$. If $\alpha\leq0$ the maximum over
$\v$ is zero and \eqref{eq:pde} would give $u(0)=0$, contradicting
$u(0)>0$; hence $\alpha>0$, the maximum is $6\alpha/5$, attained for
$|\v|=2,3$, and \eqref{eq:pde} at the origin gives $u(0)=3\alpha/5$. Taylor's
formula with $\alpha=\tfrac53u(0)$ and
$x^\top(\mathrm{Id}-\tfrac15\one\one^\top)x=\tfrac45\bigl(\sum_ix_i^2-\tfrac12\sum_{i<j}x_ix_j\bigr)$
is \eqref{eq:origin-expansion}, with $u(0)=45\pi^2/(512\sqrt2)$ by
Lemma~\ref{lem:e0} and $\tfrac23u(0)=15\pi^2/(256\sqrt2)$. The same
argument with four experts gives
$\nabla^2u_4(0)=2u_4(0)(\mathrm{Id}-\tfrac14\one\one^\top)$ and recovers
\cite[equation~(3.4)]{BEZ20} from $u_4(0)=\pi/(4\sqrt2)$.
\end{proof}

We now prove Theorem~\ref{thm:comb}. The Hamiltonian inequalities of this
section give $\Delta\geq0$ for the gap \eqref{eq:comb-gap} but do not
identify where equality holds. The proof upgrades this nonnegativity to
strict positivity: in Regions I and II through an exact positive-kernel
formula for $\Delta$, and in Region III through the strict positivity of
the generators of \eqref{eq:eight-generators}.

\begin{proof}[Proof of Theorem~\ref{thm:comb}]
We treat Regions I and II first, through an exact formula for the gap.
Both coordinate matrices \eqref{eq:q12} send $b_{\v_*}=(1,-1,1,0)$ to
$(0,0,1,0)$ and $b_{\v_C}=(1,-1,1,-1)$ to $(0,0,1,-1)$ in the coordinates
$(z_1,d,z_2,z_4)$, $d=|z_3|$. By \eqref{eq:direction-scaling},
$D_{\v_*}^2u=(2/k)F_{z_2z_2}$ and
$D_{\v_C}^2u=(2/k)(F_{z_2z_2}-2F_{z_2z_4}+F_{z_4z_4})$, hence
\begin{equation}\label{eq:comb-differential}
 \Delta=\frac2k(2F_{z_2z_4}-F_{z_4z_4}).
\end{equation}
Since the four-expert contribution is independent of $z_4$, it cancels
from the gap. Neither direction changes $z_1$, so differentiating the
integral produces no endpoint term. With $c=\coth t$, the integrand of
\eqref{eq:formula12} carries the factor $\eu^{-2z_4c}\Phi_t(z_2)$, on which
$\partial_{z_4}$ produces $-2c$ and
$\partial_{z_2}\Phi_t(z_2)=-N_t(z_2)$ with $N_t(z_2)=c\cosh z_2-\sinh z_2$;
so $2\partial_{z_2}\partial_{z_4}-\partial_{z_4}^2$ replaces
$\Phi_t(z_2)$ by $4c\,[N_t(z_2)-c\,\Phi_t(z_2)]$, and the identity
\[
 (c\cosh z_2-\sinh z_2)
  -c(\cosh z_2-c\sinh z_2)=(c^2-1)\sinh z_2
\]
reduces the differentiated kernel to a positive one:
\begin{equation}\label{eq:comb-positive}
 \Delta=\frac8k\sinh z_2\int_{z_1}^\infty
  \sinh t\,p(t)\coth t(\coth^2t-1)\eu^{-2z_4\coth t}
                  \Phi_t(z_1)\Phi_t(d)\dd.
\end{equation}
For $z_1>0$, every factor in the integrand is strictly positive for $t>z_1$.
Therefore $\Delta=0$ exactly when $z_2=0$, or $y_1=y_3=0$.
If $z_1=0$, the first four coordinates coincide. The controls $\v_*=10100$
and $\one-\v_C=01010$ each select two of those four tied coordinates, so
their curvatures agree by permutation symmetry, and $\v_C$ has the same
curvature as its complement. This proves the theorem in
Regions I and II, including their boundary collisions.

In Region III we use the multiaffine structure of the gap. For $a_4>0$ put
\[
 \zeta_i=\frac{\tanh a_i}{\tanh a_4}\in[0,1],\qquad
 \widehat G=\frac{k\Delta}{2\cosh a_1\cosh a_2\cosh a_3}.
\]
The COMB direction in these coordinates is $q_{\one-\v_C}=(-2,1,-2,1)/3$, the image of
$b_{\one-\v_C}=(-1,1,-1,1)$ under \eqref{eq:q3}; its sign is
immaterial in the quadratic form.
The eight values of its multiaffine gap are
\begin{center}
\begin{tabular}{ccccccccc}\toprule
$(\zeta_1,\zeta_2,\zeta_3)$ & 000 & 001 & 010 & 011 & 100 & 101 & 110 & 111\\\midrule
$\widehat G$ & 0 & $D_1$ & 0 & $B_2$ & $D_1$ & $A_2$ & $B_2$ & $P_3$\\\bottomrule
\end{tabular}
\end{center}
The coefficients are the generators defined in
\eqref{eq:eight-generators}; permuting the three $a_i$ explains the repeated
values.

To characterize the equality set, we need to show that these generators
are strictly positive.
For each $G\in\{D_1,B_2,A_2,P_3\}$ its first-order certificate has
$a>0$ and $\mathcal R=(G/(aC))'\leq0$. In each of the seven polynomial certificates
for $-\mathcal R$, at least one Bernstein coefficient is strictly positive and the
others are nonnegative. Every Bernstein basis function is positive in the
open interval. Thus $\mathcal R<0$ in each interval interior. Since $G/(aC)\to0$,
integration from infinity gives
\begin{equation}\label{eq:strict-generators}
 D_1(L)>0,\qquad B_2(L)>0,\qquad A_2(L)>0,\qquad P_3(L)>0
 \quad(L>0).
\end{equation}
The strict inequalities follow from 28 of the 56 exact polynomial
certificates already used in the Region III verification.

Using multiaffine interpolation of the table, grouping the vertex terms
by their dependence on $\zeta_2$, we obtain the factorization
\begin{equation}\label{eq:comb-factor}
\begin{aligned}
 \widehat G={}&(\zeta_1+\zeta_3-2\zeta_1\zeta_3)
                    \bigl[(1-\zeta_2)D_1+\zeta_2B_2\bigr]\\
              &+\zeta_1\zeta_3\bigl[(1-\zeta_2)A_2+\zeta_2P_3\bigr].
\end{aligned}
\end{equation}
Both brackets are strictly positive. Also
$\zeta_1+\zeta_3-2\zeta_1\zeta_3=\zeta_1(1-\zeta_3)+(1-\zeta_1)\zeta_3\geq0$.
Consequently the gap vanishes exactly when $\zeta_1=\zeta_3=0$.
The physical ordering $0\leq a_1\leq a_2\leq a_3$ forces all three $a_i$
to vanish. By \eqref{eq:inverse3}, this is $y_1=y_3=y_4=0$.
It is precisely the part of \eqref{eq:comb-set} lying in Region III.
The case $a_4=0$ is the full collision and follows by continuity.
Together with \eqref{eq:comb-positive}, this completes the proof.
\end{proof}

\appendix
\section{Proof of the gluing lemma}\label{sec:appendix-gluing}

We prove Lemma~\ref{lem:gluing} in the notation of its statement:
$\Omega\subset\R^n$ is open and convex, $K_1,\dots,K_N$ are closed convex
sets with nonempty interiors whose union contains $\Omega$, and $u$ is
continuous on $\Omega$, is $C^2$ in the interior of each $K_j$, and has a
gradient and Hessian that extend continuously to each $\Omega\cap K_j$ and
agree at every point of $\Omega$ for all the $K_j$ containing it.

\begin{proof}[Proof of Lemma~\ref{lem:gluing}]
Let $g$ and $h$ be the common extended gradient and Hessian. They are well
defined by the agreement hypothesis and continuous on $\Omega$, since each
is continuous on the finitely many relatively closed sets $\Omega\cap K_j$
covering $\Omega$. Fix $x,y\in\Omega$; the segment $[x,y]$ lies in
$\Omega$. Each intersection $[x,y]\cap K_j$ is a closed subsegment, and
these finitely many subsegments cover $[x,y]$, so taking all their
endpoints as partition points gives $0=s_0<s_1<\dots<s_m=1$ such that
every segment $[x+s_{i-1}(y-x),\,x+s_i(y-x)]$ lies in a single $K_j$. For
a segment $[a,b]$ in the interior of $K_j$, the fundamental theorem of
calculus gives
\begin{equation}\label{eq:segment-integral}
 u(b)-u(a)=\int_0^1g\bigl(a+s(b-a)\bigr)\cdot(b-a)\,\diff s.
\end{equation}
If $[a,b]\subset K_j$ meets the boundary of $K_j$, choose a point $z$ of
$\Omega$ in the interior of $K_j$, which exists because $\Omega$ is open
and meets $K_j$. For $0<\varepsilon\leq1$ the segment with endpoints
$(1-\varepsilon)a+\varepsilon z$ and $(1-\varepsilon)b+\varepsilon z$ lies
in $\Omega$ and in the interior of $K_j$ by convexity, so
\eqref{eq:segment-integral} holds for it, and letting $\varepsilon\to0$
gives \eqref{eq:segment-integral} for $[a,b]$, by the continuity of $u$
and $g$ on $\Omega\cap K_j$. Summing over the subsegments,
\[
 u(y)-u(x)=\int_0^1g\bigl(x+s(y-x)\bigr)\cdot(y-x)\,\diff s,
\]
so $u(y)-u(x)-g(x)\cdot(y-x)=o(|y-x|)$ as $y\to x$, by the continuity of
$g$. Thus $u$ is differentiable with $\nabla u=g$, and $u\in C^1(\Omega)$.
The same argument applied to $g$, whose restriction to the interior of
each $K_j$ is $C^1$ with derivative $h$, gives $\nabla g=h$, so
$u\in C^2(\Omega)$.
\end{proof}

\bibliographystyle{abbrv}
\bibliography{main}

\end{document}